\documentclass[a4paper,reqno,oneside]{amsart}

\usepackage[
style=numeric,
sorting=nty,
maxnames=99,
maxalphanames=5,
sortcites,
date=year
]{biblatex}

\renewbibmacro*{volume+number+eid}{%
  \printfield{volume}%
  \printtext{ Vol }%
  \printfield{number}%
  \printtext{}%
  \setunit{\addcomma\space}%
  \printfield{eid}}

\usepackage[british]{babel}
\usepackage{csquotes}  
\numberwithin{equation}{section}
\usepackage{amssymb}
\usepackage[protrusion=true,expansion=true]{microtype}	
\usepackage{amsmath,amsfonts,amsthm} 
\usepackage[pdftex]{graphicx}	
\usepackage{url}
\usepackage{svg}
\usepackage[title,titletoc,page]{appendix} 
\usepackage{listings} 
\usepackage{color}
\usepackage{caption}
\usepackage{subcaption}
\usepackage{algorithm}
\usepackage{dsfont} % for 1 bbmath
\usepackage{bm}
\usepackage{upgreek}
\usepackage[normalem]{ulem}
\usepackage{enumerate}
\usepackage{mathtools}
\usepackage{lipsum}
\usepackage{adjustbox}
\usepackage{tikz}
\usepackage{circuitikz}
\usepackage{tcolorbox}
\usepackage{xargs}      % Use more than one optional parameter

\usepackage{todonotes}
\usepackage{svg}
\usepackage{nicematrix}

\usepackage[T1]{fontenc} % To switch to the T1 encoding
\usepackage{lmodern} % To switch to Latin Modern
\rmfamily 
\DeclareFontShape{T1}{lmr}{b}{sc}{<->ssub*cmr/bx/sc}{}
\DeclareFontShape{T1}{lmr}{bx}{sc}{<->ssub*cmr/bx/sc}{}

\usepackage[left=2.4cm,right=2.4cm,top=2.9cm,bottom=3.1cm,footskip=1.5cm,headsep=1cm]{geometry}
\usepackage[%
pdfauthor={Yannick De Bruijn, Erik Orvehed Hiltunen},%
pdftitle={Open Limit of banded Toeplitz matrices},
pdfsubject={Open Limit of banded Toeplitz matrices}]{hyperref}

\usepackage{orcidlink}
\usepackage{fancyhdr}
\fancypagestyle{plain}{%
\fancyhead[C]{}
\fancyfoot[C]{}

}

\fancypagestyle{nosection}{%
\fancyhead[CE]{}
\fancyhead[CO]{}
\fancyfoot[CE,CO]{\thepage}
}
\numberwithin{equation}{section}		% Equationnumbering: section.eq#
\numberwithin{figure}{section}			% Figurenumbering: section.fig#
\numberwithin{table}{section}			% Tablenumbering: section.tab#

\newcommand{\vect}[1]{\boldsymbol{\mathbf{#1}}}

\DeclareMathOperator{\C}{\mathbb{C}}

\newcommand{\N}{\mathbb{N}}

\newcommand{\R}{\mathbb{R}}

\newcommand{\Z}{\mathbb{Z}}

\renewcommand{\epsilon}{\varepsilon}

\DeclareMathOperator{\dos}{\mathrm{DoS}}

\renewcommand{\tilde}{\widetilde}
\renewcommand{\hat}{\widehat}
\renewcommand{\i}{\mathrm{i}}
\renewcommand{\d}{\,\mathrm{d}}

\usepackage[noabbrev,capitalize]{cleveref}
\crefname{proposition}{Proposition}{Propositions}
\crefname{equation}{}{}

\newtheorem{theorem}{Theorem}[section]
\newtheorem{lemma}[theorem]{Lemma}
\newtheorem{proposition}[theorem]{Proposition}
\newtheorem{corollary}[theorem]{Corollary}
\newtheorem{conjecture}[theorem]{Conjecture}

\theoremstyle{definition}
\newtheorem{definition}[theorem]{Definition}
\newtheorem{assumption}[theorem]{Assumption}

\newtheorem{remark}[theorem]{Remark}

\crefname{assumption}{Assumption}{Assumptions}
\crefname{definition}{Definition}{Definitions}
\crefname{corollary}{Corollary}{Corollaries}
\crefname{enumi}{item}{items}

\begin{document}
\title[Mathematical Foundation for the Generalised Brillouin zone of m-banded Toeplitz operators.]{Mathematical Foundation for the Generalised Brillouin zone of m-banded Toeplitz operators.}

\author[Y. de Bruijn]{Yannick de Bruijn$^*$ \,\orcidlink{0009-0009-6194-0761}}
\address{\parbox{\linewidth}{Yannick de Bruijn\\
Department of Mathematics, University of Oslo, Moltke Moes vei 35, 0851 Oslo, Norway.
\\
\href{https://orcid.org/0009-0009-6194-0761}{orcid.org/0009-0009-6194-0761}}}
\email{yannicd@math.uio.no}

\author[E. O. Hiltunen]{Erik Orvehed Hiltunen \,\orcidlink{0000-0003-2891-9396}}
\address{\parbox{\linewidth}{Erik Orvehed Hiltunen\\
Department of Mathematics, University of Oslo, Moltke Moes vei 35, 0851 Oslo, Norway. \\
\href{http://orcid.org/0000-0003-2891-9396}{orcid.org/0000-0003-2891-9396}.}}
\email{erikhilt@math.uio.no}

%%%-------------------------------------------------------
\begin{abstract}
    We show that the open-boundary limit of a banded Toeplitz matrix is real whenever its two middle roots trace out a closed $C^1$ polar curve, with the remaining roots strictly separated in modulus. Building on this result, we develop both analytical and numerical methods to asymptotically symmetrise a class of banded non-Hermitian Toeplitz matrices whose asymptotic spectra are real. Finally, we provide a rigorous mathematical foundation for the generalised Brillouin zone, a concept widely used in non-Hermitian physics, by proving that it coincides with the polar curve on which the symbol function takes real values.
\end{abstract}
\maketitle
\vspace{3mm}
\noindent
\textbf{Mathematics Subject Classification (MSC2020): } 15A18, %Eigenvalues, singular values, and eigenvectors
15B05, %Toeplitz, Cauchy, and related matrices,
65F15  %Numerical computation of eigenvalues and eigenvectors of matrices

\vspace{3mm}

\vspace{5mm}

\section{Introduction}

Toeplitz matrices are a class of structured matrices in which each diagonal is filled with the same value. These matrices frequently appear as Hamiltonians that model systems with a periodic structure, for instance, periodic media or metamaterials.
The so called \emph{open limit} of the spectrum of Toeplitz matrices, see \cite{BookSpectraBandedToepllitz} for an overview, is of particular interest as it corresponds to the asymptotic spectrum of a finite Toeplitz matrix which is growing in dimension. For non-Hermitian matrices, it is widely understood that the open spectrum cannot be described using conventional Floquet-Bloch band theory. In the specific case of tridiagonal Toeplitz matrices, it was demonstrated in \cite{ammari2024generalisedbrillouinzonenonreciprocal,debruijn2025complexbrillouinzonelocalised} that the open limit is instead accurately captured by complex Floquet-Bloch boundary conditions, where the imaginary part of the quasimomentum is constant and non-zero. In the more general $m$-banded case, the imaginary part of the Bloch quasimomentum may vary across the spectrum and is specified by the generalised Brillouin zone \cite{ammari2024generalisedbrillouinzonenonreciprocal,ashida.gong.ea2020NonHermitian, PhysRevLett.121.086803, PhysRevLett.125.126402, Okuma2020}. Although the generalised Brillouin zone is a widely used tool in the physics literature, a number of mathematical questions remain unanswered. 

An open problem to this date is to find a criterion on the symbol function that guarantees the open limit to be real-valued. Although \cite{RealSpectrumToeplitz} previously addressed this question, the subsequent corrigendum \cite{CorrectionRealAsymptoticSpectrum}, by the same authors, identified an error in one of their proofs, thereby leaving the problem open. In this manuscript, we partially address this gap by establishing a result under stronger assumptions than those stated in \cite{RealSpectrumToeplitz}.
We prove that if it is possible to find a closed polar curve (i.e. a Jordan curve which can be parametrised by the angle of the polar coordinates) along which the symbol function evaluates to real numbers, the open limit will be real.
Evaluating the symbol function along this polar curve, rather than the standard unit circle, produces a transformed Toeplitz matrix which will be Hermitian. In particular, we will demonstrate that both the open limit and the density of states remain unchanged under this transformation. Furthermore, we illustrate numerically that the eigenvalues converge pointwise, which demonstrates that the symbol transform can be seen as an asymptotic similarity transform between a class of non-Hermitian and Hermitian Toeplitz matrices.

We revisit some concepts from Toeplitz theory that were known since the 1960s \cite{ToeplitzOriginalPaper, HirschmannSpectra, bams/1183529098}, in particular, the results from \cite{Schmidt_Spitzer_1960, Widom, bams/1183529098} will allow us to define the correct polar curve on which we evaluate the symbol function.
We illustrate that the polar curve is the natural generalisation of complex band structure in tridiagonal Toeplitz matrices \cite{debruijn2025complexbandstructurelocalisation, debruijn2025spectrapseudospectranonhermitiantoeplitz, ammari2024spectra}. 
Moreover, we will make an important connection between this polar curve and generalised Brillouin zone, and demonstrate that our method has  widespread use in the asymptotic study of non-Hermitian Bloch Hamiltonians explored in condensed matter physics.

An important application of our method is in the realm of non-Hermitian physics. A prominent example is the non-Hermitian skin effect, which refers to the localisation of the bulk eigenmodes at the boundary of a non-Hermitian open system \cite{ashida.gong.ea2020NonHermitian, PhysRevLett.121.086803, PhysRevLett.125.126402, Okuma2020, PhysRevB.111.035109}.
This new framework enables to extend the tight binding Hamiltonians with non-reciprocal coupling, as introduced by Hatano and Nelson \cite{PhysRevB.58.8384, PhysRevLett.77.570} to models with long range coupling.

This paper is organised as follows. In Section \ref{sec: Toeplitz Analysis}, we recall the most important results from Toeplitz theory and present a criterion on the symbol function that guarantees the reality of the open limit. We demonstrate that the open limit remains unchanged under a deformation of the torus on which the symbol function is evaluated. Moreover, we provide numerical evidence that, in the limit of large matrix size, the two sets of eigenvalues of the finite Toeplitz matrices converge.
We prove that under the assumptions that the open limit is real, the density of states of a non-Hermitian matrix Toeplitz matrices agrees with that of the Hermitian matrix which is generated by deforming the unit torus to the generalised Brillouin zone.
In Section \ref{sec: Bloch Hamiltonian} we link the results from Section \ref{sec: Toeplitz Analysis} and provide a mathematical foundation for the generalised Brillouin zone.
In Section \ref{sec: Conclusion} we conclude on our work and present future works and possible applications of the newly developed theory. 
The Matlab code that supports the findings of this article is available in a repository linked in Section \ref{Sec: Data availability}.

\section{Toeplitz Analysis}\label{sec: Toeplitz Analysis}
This paper is devoted to the study of the asymptotic spectra of banded Toeplitz matrices. A Toeplitz operator $\mathbf{A}$ is generated by the real-valued sequence $\{a_k\}_{k \in \Z} \subset \R$ and is given by the semi-infinite matrix,
\begin{equation}\label{eq: def Toeplitz Operator}
    \mathbf{A} = \begin{pmatrix}
        a_0 & a_{-1} & a_{-2} & a_{-3} & \cdots\\
        a_1 & a_0    & a_{-1} & a_{-2} & \cdots\\
        a_2 & a_1 & a_0 & a_{-1} & \cdots \\
        a_3 & a_2 & a_1 &a_0 & \cdots\\
        \vdots & \vdots & \vdots & \vdots & \ddots
    \end{pmatrix}.
\end{equation}
Throughout this paper we assume $m\geq 1$ and define a $m$-banded Toeplitz operator as 
\begin{equation} \label{eq: def banded Toeplitz Operator}
    \mathbf{A}_m = \begin{cases}
        \mathbf{A}_{i,j}, &\text{ for } |i-j| \leq m,\\
        0, &\text{ for } |i-j| > m.
    \end{cases}
\end{equation}
We will illustrate our results on Toeplitz matrices with non-reciprocal algebraic off-diagonal decay, that is
\begin{equation}
    \begin{cases}
        a_k = \mathcal{O}(|k|^{-\kappa}),~ k > 0,\\
        a_k = \mathcal{O}(|k|^{-\rho}),~k < 0.
    \end{cases}
\end{equation}

A Toeplitz operator is a bounded operator on $\ell^2$ if and only if it is generated by a function $f\in L^\infty$ \cite{ToeplitzOriginalPaper}.
This generating function is commonly referred to as the \emph{symbol function} as its Fourier coefficients generate the sequence $\{a_k\}_{k = -\infty}^\infty$, that is
\begin{equation}\label{eq: def symbol function operator}
    f(z) = \sum_{k = -\infty}^\infty a_k z^{-k}.
\end{equation}
For an $m$-banded Toeplitz operator \eqref{eq: def banded Toeplitz Operator}, the symbol function is given by the finite Laurent polynomial,
\begin{equation}\label{eq: def symbol function banded}
    f_m(z) = \sum_{k = - m}^m a_k z^{-k},\quad a_{\pm m}\neq 0.
\end{equation}
In the sequel of this paper, we shall refer to the Toeplitz operator $\vect{A}$, specified in equation \eqref{eq: def Toeplitz Operator}, by the notation $\vect{A} = \vect{T}(f)$ and for the banded Toeplitz operator \eqref{eq: def banded Toeplitz Operator}, by $\mathbf{A}_m = \mathbf{T}(f_m)$. We note that the symbol function $f_m$ is a special case of $f$, obtained by setting $a_j = 0$ for all $|j| > m$.

We now define the region of non-trivial winding. We let $\mathbb{T}$ denote the one dimensional unit circle in $\mathbb{C}$, that is $\mathbb{T} = e^{-\i[0, 2\pi]}$, then
\begin{equation}\label{def: winding region}
     \sigma_\text{wind} := \bigl\{\lambda \in \C \setminus f(\mathbb{T}) : \operatorname{wind} \bigl(f, \lambda \bigr) \neq 0 \bigr\},
\end{equation}
where the winding is defined as
\begin{equation}
    \operatorname{wind}(f,\lambda) := \frac{1}{2\pi \i} \int_{\mathbb{T}} \frac{f'(z)}{f(z)-\lambda}\,dz, \quad \lambda \in \mathbb{C}\setminus f(\mathbb{T}).
\end{equation}
The spectrum of a Toeplitz operator is given by the following result \cite[Theorem 1.17]{LargeTruncatedToeplitz}.

\begin{theorem}[Gohberg]\label{thm: Gohberg Spectrum for baned Toeplitz operator}
    The operator $\mathbf{T}(f)$ is Fredholm on $\ell^2$ if  and only if $f(e^{\i\alpha}) \neq 0, ~ \forall \alpha \in [0, 2\pi)$, in which case
    \begin{equation}
        \operatorname{Ind}\mathbf{T}(f) = - \operatorname{wind}(f, 0 )
    \end{equation}
    and the spectrum is given by
    \begin{equation}
        \sigma\bigl(\mathbf{T}(f)\bigr) =  f(\mathbb{T}) \cup \sigma_\mathrm{wind},
    \end{equation}
    where $\sigma_\text{wind}$ was defined in \eqref{def: winding region}.
\end{theorem}

For later application, we will introduce the spectra of large truncated Toeplitz operators.
To define a finite Toeplitz matrix, we introduce the projection operator,
\begin{align}
    \vect{P}_n: \ell^2(\N, \mathbb{C}) &\to \ell^2(\N, \mathbb{C})\\
    (x_1, x_2, x_3, \dots) &\mapsto (x_1, \dots, x_n, 0, 0, \dots).
\end{align}
A finite Toeplitz matrix is derived from a Toeplitz operator by performing the following truncation,
\begin{equation}\label{eq: truncate operator to matrix}
    \vect{T}_{n}(f) := \vect{P}_{n} \vect{T}(f) \vect{P}_{n}.
\end{equation}
We will be particularly interested in finding the smallest set on which the eigenvalues will cluster as the matrix dimension becomes large. The limiting spectrum of a finite but growing Toeplitz matrix will be referred to as the \emph{open limit} and is defined in the Hausdorff sense as
\begin{equation}\label{def: open spectrum}
    \sigma_{\mathrm{open}}\bigl(\mathbf{T}(f)\bigr) := \lim_{n \to\infty}\sigma\bigl(\mathbf{T}_n(f)\bigr).
\end{equation}
A useful characterisation of the limiting set has been presented in \cite[Section 5.8]{LargeTruncatedToeplitz}, with the foundational paper being \cite{Schmidt_Spitzer_1960}.

\begin{theorem}[Schmidt-Spitzer]\label{thm: schmidt spitzer magnitude}
    Let $f_m$ be a Laurent polynomial as defined in \eqref{eq: def symbol function banded}, then
    \begin{align}
        \lim_{n \to\infty}\sigma\bigl(\mathbf{T}_n(f_m)\bigr) &= \bigcap_{r \in (0, \infty)} \sigma\Bigl(\mathbf{T}\bigl(f_m(r\mathbb{T})\bigr)\Bigr)\\
        &= \bigl\{ \lambda \in \C ~:~ |z_{m}(\lambda)| = |z_{m+1}(\lambda)| \bigr\} \label{eq: modulus definition},
    \end{align}
    where $z_i(\lambda)$ are the roots of the polynomial equation $z^m\bigl(f_m(z)-\lambda\bigr)= 0$ sorted in ascending magnitude, that is, $0< |z_1(\lambda)| \leq |z_2(\lambda)| \leq \dots \leq |z_{2m}(\lambda)|$.
\end{theorem}

Tracking the magnitude of the middle root also plays a central role in pseudoeigenvectors construction as roots with $|z_{m+1}(\lambda)|<1$ generate exponentially localised pseudoeigenvectors. An explicit construction is given in \cite{debruijn2025spectrapseudospectranonhermitiantoeplitz} and a visualisation can be found in Figure \ref{Fig: Hamiltonian crossing}.
The representation of the open limit given by \eqref{eq: modulus definition} allows to deduce topological and analytical properties \cite{doi:10.1137/23M1587804}. This set equals the union of a finite number of pairwise disjoint open analytic arcs and a finite number of the so-called \emph{degenerate points}. We say that $\lambda \in \bigl\{ \lambda \in \C ~:~ |z_{m}(\lambda)| = |z_{m+1}(\lambda)| \bigr\}$ is a degenerate or exceptional point if the polynomial $f_m(z)-\lambda$ has confluent roots. It is not hard to see that this is the case for at most $2m$ discrete roots which are characterised by the roots of $f_m'(z)= 0$ and therefore represent a set of measure $0$ in the complex plane. The connectedness of the spectrum of Toeplitz operators was first established by Widom in \cite{Widom}. A similar result for the open limit was first demonstrated in \cite{bams/1183529098}.
\begin{theorem}[Ullman]\label{thm: connectedness of open limit}
    Let $f_m$ be a Laurent polynomial, then $\sigma_{\mathrm{open}}\bigl(\mathbf{T}(f_m)\bigr)$ is a connected set.
\end{theorem}

In particular as $\sigma_{\mathrm{open}}\bigl(\mathbf{T}(f_m)\bigr)$ consists of a finite union of analytic arcs, it directly follows from Theorem \ref{thm: connectedness of open limit} that $\sigma_{\mathrm{open}}\bigl(\mathbf{T}(f_m)\bigr)$ is path-connected.

\subsection{Criteria for real asymptotic  spectra of non-Hermitian Toeplitz matrices.}\label{Sec: open limit real spectra}
An important class of non-Hermitian Toeplitz matrices includes those with real spectra, often termed pseudo-Hermitian Toeplitz matrices. In this section, our focus will be on examining conditions under which the open limit, that is, finite matrices approaching infinite size, exhibit real spectra. 

\begin{definition}
    The symbol function $f_m$ is \emph{collapsed at $p\in \C$} if $p = f_m(re^{\i\alpha_1}) = f_m(re^{\i\alpha_2})$ for some $ r>0$ and $\alpha_1\neq\alpha_2$.
\end{definition}
Consider the following polynomial
\begin{equation}\label{eq: polynomial symbol}
    z^m\bigl(f_m(z) - \lambda \bigr) = \prod_{i = 1}^{2m}\bigl(z - z_i(\lambda) \bigr)
\end{equation}
and suppose the roots are sorted in ascending magnitude,
\begin{equation}
    |z_1(\lambda)| \leq \dots \leq |z_m(\lambda)| \leq |z_{m+1}(\lambda)| \leq \dots \leq |z_{2m}(\lambda)|.
\end{equation}
Let us define the following set,
\begin{equation}\label{set: non confluent CBS}
    \mathbf{Z}(f_m) := \bigl\{ z_m(\lambda), z_{m+1}(\lambda) \in \C~:~ |z_m(\lambda)| = |z_{m+1}(\lambda)|, \text{ for some } \lambda \in \R \bigr\}.
\end{equation}
It is easy to see that $\mathbf{Z}(f_m) \subseteq f_m^{-1}(\R)$, where $f_m^{-1}(\R)$ denotes the set of preimages or roots of $f_m(z)-\lambda= 0$ for $\lambda\in \R$. Moreover, it follows from the definition of the set in \eqref{set: non confluent CBS} that,
\begin{equation}
    \bigl\{ \lambda \in \R ~:~ |z_{m}(\lambda)| = |z_{m+1}(\lambda)| \bigr\} = \big\{ f_m(z) ~:~z \in \mathbf{Z}(f_m) \big \}.
\end{equation}

A priori, it is possible for some of the roots in \eqref{eq: polynomial symbol} to become confluent. Such roots are characterised by the fact that they are also roots of $f_m'(z)$. Since $f'_m$ has at most $2m$ roots, there are at most $2m$ confluent roots independent of $\lambda$. We present the following standard assumption \cite{ComputeBandedToeplitz}.

\begin{assumption}\label{Confluence Assumption}
    Let $f_m$ be as in \eqref{eq: def symbol function banded}. We assume that for every $\lambda \in \R$ with $|z_m(\lambda)| = |z_{m+1}(\lambda)|$ the remaining moduli are strictly separated,
    \begin{equation}\label{eq: separation}
        |z_{m-1}(\lambda)| < |z_m(\lambda)| = |z_{m+1}(\lambda)| < |z_{m+2}(\lambda)| .
    \end{equation}
\end{assumption}
The structure of the roots yields the following lemma.
 
\begin{lemma}\label{lem: separation implies non confluence}
    Suppose Assumption \ref{Confluence Assumption} holds. Then $f_m'(z) \neq 0$ for every $z \in \mathbf{Z}(f_m)\setminus\R$, and $f''_m(z) \neq 0$ for every $z \in \mathbf{Z}(f_m)\cap\R$ such that $f_m'(z) = 0$.
\end{lemma}
 
\begin{proof}
    Let $z \in \mathbf{Z}(f_m)$ and put $\lambda := f_m(z) \in \R$, so that $|z| = |z_m(\lambda)| = |z_{m+1}(\lambda)|$. Roots of equal modulus occupy consecutive positions in the ordering, and the block of roots of modulus $|z|$ contains the positions $m$ and $m+1$; by \eqref{eq: separation} it contains no further position, hence it consists of exactly two roots.
 
    If $z \notin \R$ and $f_m'(z) = 0$, then $z$ is a root of $f_m - \lambda$ of multiplicity at least two and, since $f_m$ has real coefficients and $\lambda \in \R$, so is $\overline{z} \neq z$. The block would then contain at least four roots, a contradiction. If $z \in \R$ and $f_m'(z) = f''_m(z) = 0$, then $z$ is a root of multiplicity at least three and the block would contain at least three roots, again a contradiction.
\end{proof}

Moreover, the set $\mathbf{Z}(f_m)$ together with Assumption \ref{Confluence Assumption} avoids confluent roots and, at first glance, might seem difficult to visualise, but has an intuitive visual representation when plotting the complex band structure. Namely, if the main complex band does not touch or intersect other complex bands, see Figure \ref{Fig: Lambda spanned by polar curve} (A), which is a standard condition avoiding exceptional points \cite{ComputeBandedToeplitz}. 
We have another important observation on the structure of the set $\mathbf{Z}(f_m)$ under Assumption \ref{Confluence Assumption}.
\begin{remark}\label{rem: real and imag confluent roots}
    Suppose that Assumption \ref{Confluence Assumption} is satisfied, then the set $\mathbf{Z}(f_m)$ cannot contain complex confluent roots, because the roots of $f_m(z)-\lambda$ appear in complex-conjugate pairs. However, the set $\mathbf{Z}(f_m)$ may contain real-valued confluent roots, as long as $f''_m(z) \neq 0$ for $z\in \R$.
\end{remark}
In the sequel, we will introduce some more structure on the geometry of the set $\mathbf{Z}(f_m)$.

\begin{definition}\label{def: polar curve}
    We define a $C^1$ \emph{polar curve} as the set parametrised by polar coordinates such that
    \begin{equation}\label{eq: def eq of polar curve}
        p(e^{\i\theta}) = r(\theta) e^{\i \theta}, \quad \theta \in [-\pi, \pi],
    \end{equation}
    for some continuously differentiable function $r: [-\pi, \pi] \to (0, \infty)$ describing the radius as function of the angle $\theta$, such that $r(-\theta) = r(\theta)$ and $r'(\pi) = r'(-\pi)$. For ease of notation, we will often denote the polar curve as $p(\mathbb{T})$, where $\mathbb{T}$ is the one-dimensional torus.
\end{definition}
The next result relates the geometry of the set $\mathbf{Z}(f_m)$ to the analytical properties of the complex band structure.

\begin{lemma}\label{lem: monotone alpha}
    If the set $\mathbf{Z}(f_m)$ is spanned by a polar curve and Assumption \ref{Confluence Assumption} holds, then $\lambda(\alpha)$ is strictly monotone in $\alpha$ on the sets $(-\pi, 0)$ and $(0, \pi)$.
\end{lemma}
\begin{proof}
    We start by noting that $\lambda(\alpha) = f_m(e^{\i(\alpha + \i\beta(\alpha))})$. Therefore, it follows that
    \begin{equation}
        \frac{\d}{\d\alpha}\lambda(\alpha) = \frac{\d}{\d\alpha}f_m(e^{\i(\alpha + \i\beta(\alpha))}) = f'_m(e^{\i(\alpha + \i\beta)})\Big(e^{\i(\alpha + \i\beta)}\i\big( \underbrace{1+\i \beta'(\alpha)}_{\neq0}\big)\Big).
    \end{equation}
    By Assumption \ref{Confluence Assumption} for any element $z\in \mathbf{Z}(f)$ it follows $f'_m(z) = 0$ only if $z \in \R$, correspondingly $\alpha \in \{0, \pm\pi\}$. As a consequence, $\lambda(\alpha)$ is strictly monotone in the intervals $(-\pi, 0)$ and $(0, \pi)$. 
\end{proof}

By the inverse function theorem, restricting to either $(-\pi, 0)$ or $(0, \pi)$, the function $\alpha(\lambda)$ is strictly monotone.
As a consequence of Lemma \ref{lem: monotone alpha}, in the case where $\mathbf{Z}(f_m)$ is spanned by a polar curve the set $\mathbf{Z}(f_m)$ admits a parametrisation by the complex band structure, that is, $p(e^{\i\alpha}) = e^{\i(\alpha + \i\beta(\alpha))}$ and it follows that,
\begin{equation}\label{eq: polar curve bijection}
    \mathbf{Z}(f_m) = \bigcup_{\alpha \in[-\pi, \pi]} p(e^{\i\alpha}).
\end{equation}

The motivation behind this polar curve $p(\mathbb{T})$ is the following.
For tridiagonal Toeplitz operators, as demonstrated in \cite{ammari2024spectra, debruijn2025complexbandstructurelocalisation}, a unique $\tilde{r}$ exists such that $\bigcap_{r>0} \sigma\bigl(\mathbf{T}\bigl(f(r\mathbb{T})\bigr)\bigr) = \sigma\bigl(\mathbf{T}\bigl(f(\tilde{r}\mathbb{T})\bigr)\bigr)$. This arises because in the bulk, the complex Floquet parameter is fixed to the rate of non-reciprocity $\tilde{r} = e^{-\beta}$ for constant $\beta$. However, for general $m$-banded Toeplitz operators, this condition does not hold, that is, there is no single $r$  for which the symbol function $f(r\mathbb{T})$ will be collapsed onto $\sigma_{\mathrm{open}}$. Subsequently, it is not enough to evaluate the symbol function on the scaled torus and we will allow more general homotopic deformations of the unit circle.
Along the polar curve $p(\mathbb{T})$, we have the following result.
\begin{lemma}\label{lem: Collapse}
    For any $z \in p(\mathbb{T})$, such that $z \in \C\setminus\R$, the symbol $f_m(z)$ is collapsed.
\end{lemma}
\begin{proof}
Since the complex band structure is defined for $\lambda\in\R$ and since the coefficients of the Laurent polynomial $f_m$ are real, for any root $z$ of $f_m(z)-\lambda = 0$ it also holds that $f_m(\overline{z})-\lambda = 0$.
Since by assumption $z\in \C\setminus\R$ it must hold that  $\arg(z)\neq\arg(\overline{z})$, which completes the proof.
\end{proof}

The following result is one of our principal theorems, providing a sufficient condition that guarantees the reality of the spectrum of the open limit.

\begin{theorem}\label{Thm: real sectra criterion}
    Suppose that $\mathbf{Z}(f_m)$ is spanned by a polar curve $p(\mathbb{T})$ that is $C^1$ defined in \eqref{eq: polar curve bijection} and suppose that Assumption \ref{Confluence Assumption} is met, then it holds that
    \begin{equation}\label{eq: spectral limit}
        \lim_{n\to\infty} \sigma\bigl(\mathbf{T}_n(f_m)\bigr) = \bigcup_{\alpha \in[-\pi, \pi]} f_m\bigl(e^{\i(\alpha + \i \beta(\alpha))}\bigr) \subset\R.
    \end{equation}
\end{theorem}

\begin{proof}
    We start by demonstrating that $\sigma_{\mathrm{open}}\bigl(\mathbf{T}(f)\bigr) \subseteq \R$.
    Suppose for contradiction that there exists a point $p \in \sigma_{\mathrm{open}}\bigl(\mathbf{T}(f)\bigr)\setminus \R$.
    Consider the interval $[\lambda_1, \lambda_2] := \sigma_{\mathrm{open}}\cap\R\neq\emptyset$, and take any $q \in (\lambda_1, \lambda_2)$.
    By Theorem \ref{thm: connectedness of open limit} we know that $\sigma_{\mathrm{open}}\bigl(\mathbf{T}(f)\bigr)$ is a connected set, so there exists a continuous path $\tau:[0,1] \to \C$ such that $\tau(0) = q$ and $\tau(1) = p$. We can choose the path such that $\tau(t) \in \sigma_{\mathrm{open}}\bigl(\mathbf{T}(f)\bigr)$, $\forall t \in [0,1]$. We demonstrate that such complex branches are inadmissible under the standing assumptions.
    
 Suppose first $q \in (\lambda_1,\lambda_2)$. By monotonicity, see Lemma~\ref{lem: monotone alpha}, the values $\alpha = 0$ and $\alpha = \pi$ occur only at $\lambda_1$ and $\lambda_2$, so the middle roots of $f_m(z) - q$ form a conjugate pair $z_0 = r_0e^{\i\alpha_0}$, $\overline{z_0}$ with $z_0 \notin \R$. By Lemma~\ref{lem: separation implies non confluence} these roots are simple and, by Assumption \ref{Confluence Assumption}, strictly separated in modulus from the remaining ones. Hence they depend analytically on $\lambda$ near $q$, see \cite[Section 2]{ComputeBandedToeplitz}, and remain the middle pair on a neighbourhood of $q$. Write $z(\lambda) = r(\lambda)e^{\i\theta(\lambda)}$ for the branch with $z(q) = z_0$. Since
    \begin{equation}\label{eq: regular curve}
        r'(\lambda)^{2} + r(\lambda)^{2}\theta'(\lambda)^{2} = \left|\frac{\d z}{\d\lambda}\right|^{2} = \frac{1}{|f_m'(z)|^{2}} \neq 0 ,
    \end{equation}
    the derivatives $r'$ and $\theta'$ cannot vanish simultaneously. Were $\theta'(q) = 0$, the polar curve would satisfy $\d r/\d\theta \to \infty$ at $\alpha_0$, contradicting the definition of a $C^1$ polar curve, see Definition~\ref{def: polar curve}, hence $\theta'(q) \neq 0$. As $f_m$ has real coefficients, the other middle root is $\overline{z(\overline{\lambda})}$, so with $\delta := \lambda - q \in \C$ the branch $z$ is holomorphic near $q$ and the roots admit the following expansion
    \begin{equation}
        |z(\lambda)|^2 = |z(q+\delta)|^{2} = |z_0|^{2} + \overline{z_0}z'(q)\,\delta + z_0\overline{z'(q)}\,\overline{\delta} + \mathcal{O}(|\delta|^{2}) = r_0^{2} + 2\mathrm{Re}\bigl(\overline{z_0}z'(q)\,\delta\bigr) + \mathcal{O}(|\delta|^{2}).
    \end{equation}
    Replacing $\delta$ by $\overline{\delta}$ gives the corresponding expansion for the second branch, that is $\overline{\lambda} = q + \overline{\delta}$. Let us introduce the following functional, 
    \begin{align}\label{eq: modulus splitting}
        F(\lambda) &:= \bigl|z(\lambda)\bigr|^{2} - \bigl|z(\overline{\lambda})\bigr|^{2}\\
        &= 2\mathrm{Re}\bigl(\overline{z_0}z'(q)\,(\delta - \overline{\delta})\bigr) + \mathcal{O}(|\delta|^{2})\\
        &= -4\,\mathrm{Im}\bigl(\overline{z_0}z'(q)\bigr)\,\mathrm{Im}(\delta) + \mathcal{O}(|\delta|^{2})\\
        &= -4\,r_0^{2}\,\theta'(q)\,\mathrm{Im}(\delta) + \mathcal{O}(|\delta|^{2}).
    \end{align}
    If the perturbation $\delta$ is non-real and small enough, it follows that $F(q) \neq 0$. Since $\bigl\{|z_m(\lambda)|, |z_{m+1}(\lambda)|\bigr\} = \bigl\{|z(\lambda)|, |z(\overline{\lambda})|\bigr\}$, we have $|z_m(\lambda)| = |z_{m+1}(\lambda)|$ iff $F(\lambda) = 0$, so by \eqref{eq: modulus definition} this zero set is $\sigma_{\mathrm{open}}$ near $q$. Hence $\sigma_{\mathrm{open}}$ is real in a neighbourhood of $q$ and since the interval $(\lambda_1, \lambda_2)$ has a compact cover, there can not be a complex branch point in the interior.

    In the case $q=\lambda_1$ or $q=\lambda_2$, the middle roots $z_m(q)$ and $z_{m+1}(q)$ coincide at some $z_0\in\R\setminus\{0\}$ with $f_m'(z_0)=0$, and $f''_m(z_0)\neq0$ by Lemma \ref{lem: separation implies non confluence}. For $\lambda=q+\delta$ we write these roots as $m(\delta)\pm\eta(\delta)$, where the midpoint $m:=\tfrac12(z_m+z_{m+1})$ and the squared half-difference $\eta^{2}=\tfrac14(z_m-z_{m+1})^{2}$, being symmetric functions of the two roots contained in a small disc around $z_0$, are holomorphic in $\delta$. Note that $m(0)=z_0$, and expanding $\eta$ for small $\delta$ yields,
    \begin{equation}\label{eq: deriv of eta}
        \eta^{2}=\bigl(2/f''_m(z_0)\bigr)\delta\bigl(1+\mathcal{O}(\delta)\bigr).
    \end{equation}
    We assume that, for some small $\delta=\delta_0$, we have $|z_m(\delta_0)|=|z_{m+1}(\delta_0)|$. At this point, we have $|m+\eta|^{2}-|m-\eta|^{2} = 0$. 
    Since by the polarisation identity it holds
    \begin{equation}
        |m+\eta|^{2}-|m-\eta|^{2}=4\,\mathrm{Re}\bigl(\overline{m}\eta\bigr),
    \end{equation}
    it follows that $\mathrm{Re}\bigl(\overline{m}\eta\bigr) = 0$, which implies that $\bigl(\overline{m}\eta\bigr)^2\in (-\infty, 0]$. We introduce the holomorphic function $Q(\delta)$
    \begin{equation}
        Q(\delta):= \frac{\eta^{2}}{m^{2}} = \frac{\eta^2 \overline{m}^2}{m^2 \overline{m}^2} = \frac{\bigl(\overline{m}\eta\bigr)^2}{|m|^4},
    \end{equation}
    so that $Q(0) = 0$ and $Q(\delta_0)\in (-\infty, 0]$. Moreover, for real $\delta$, the pair of roots $\{z_m,z_{m+1}\}$ is invariant under complex conjugation, so $m$ and $\eta^{2}$ are real, and $Q(\delta)$ maps the real axis to the real axis. Moreover, $Q(\delta_0)$ is real, so that 
     \begin{equation}\label{eq:real}
     Q(\delta_0)=\overline{Q(\delta_0)}=Q(\overline{\delta_0})
     \end{equation}
     The first derivative at $\delta = 0$ is given by \eqref{eq: deriv of eta}, that is
    \begin{equation}
    Q'(0) = \frac{\bigl(\eta^{2}\bigr)'(0)\,m(0)^{2} - \eta^{2}(0)\cdot 2m(0)m'(0)}{m(0)^{4}} = \frac{\bigl(\eta^{2}\bigr)'(0)}{m(0)^{2}} = \frac{\bigl(\eta^{2}\bigr)'(0)}{z_0^{2}} = \frac{2}{z_0^{2}f''_m(z_0)} \neq 0.
\end{equation}
    Hence $Q$ is injective near $0$, so that \eqref{eq:real} $Q$ forces $\delta_0=\overline{\delta_0}\in\R$.

    The equality in \eqref{eq: spectral limit} is obtained as follows. From \eqref{eq: modulus definition} and the fact that $\sigma_{\mathrm{open}}\bigl(\mathbf{T}(f)\bigr) \subseteq \R$, we deduce that,
    \begin{align}
        \sigma_{\mathrm{open}}\bigl(\mathbf{T}(f)\bigr)&= \bigl\{ \lambda \in \R ~:~ |z_{m}(\lambda)| = |z_{m+1}(\lambda)| \bigr\}\label{eq: proof equality}\\
        &= \big\{ f_m(z)~:~z \in \mathbf{Z}(f_m) \big\}\\
        &= \bigcup_{\alpha \in[-\pi, \pi]} f_m\big(p(e^{\i\alpha})\big)\\
        &= \sigma_{\mathrm{open}}\bigl(\mathbf{T}(f_m \circ p)\bigr),
    \end{align}
    where the last equality follows from the spectral properties of Hermitian Toeplitz matrices \cite[Theorem 5.14]{LargeTruncatedToeplitz}, and the second to last equality follows from \eqref{eq: polar curve bijection}.
\end{proof}

For a visualisation of the non-continuity of $\alpha(\lambda)$, we refer the reader to Figures \ref{Fig: Lambda not spanned by polar curve} and \ref{Fig: Continuous alpha curve}.
For continuous $\alpha(\lambda)$ the proof strategy fails when the roots become confluent, which is why we require Assumption \ref{Confluence Assumption}. An illustration of a possible breaking point is provided in Figure \ref{Fig: counterexample confluent roots}. 
 
We proceed to illustrate the case where $\mathbf{Z}(f_m)$ is spanned by a polar curve in Figure \ref{Fig: Lambda spanned by polar curve}, in which case the open spectrum is real and, contrarily, an example, where $\mathbf{Z}(f_m)$ is not spanned by a polar curve in Figure \ref{Fig: Lambda not spanned by polar curve}.

\begin{figure}[h]
    \centering
    \subfloat[][Complex band structure.]
    {\includegraphics[width=0.48\linewidth]{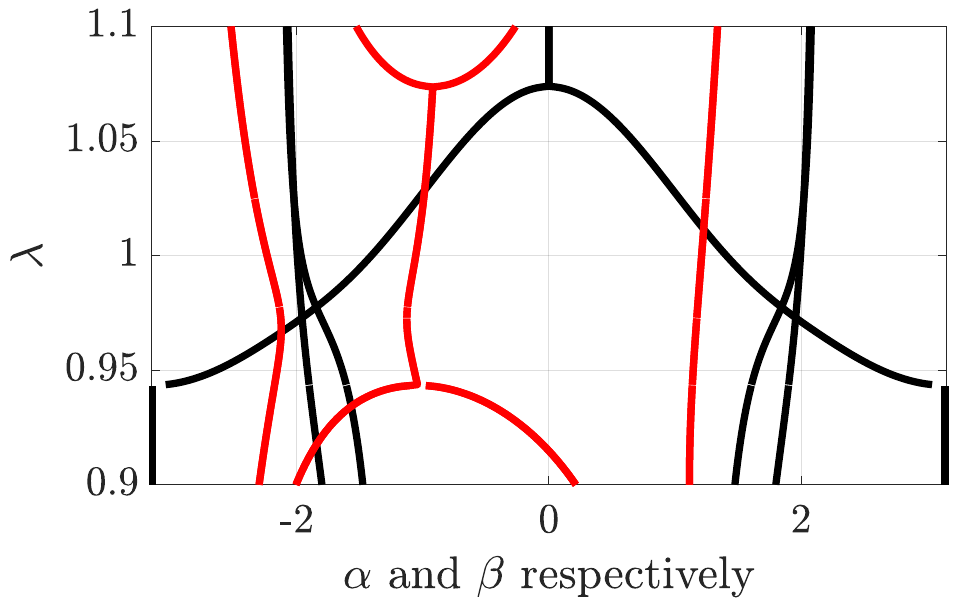}}\quad
    \subfloat[][The set $\mathbf{Z}(f_m)$ is spanned by a polar curve.]
    {\includegraphics[width=0.47\linewidth]{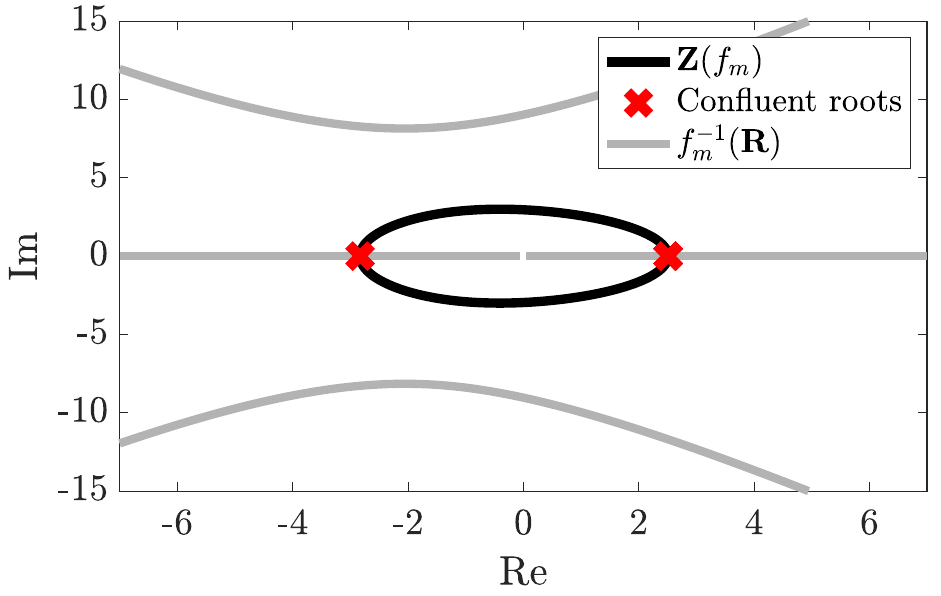}\vspace{1mm}}
    \caption{Clearly the set $\mathbf{Z}(f_m)\subseteq f_m^{-1}(\R)$ is parametrised by a polar curve. We emphasise that confluent roots are only situated on the real axis and by Remark \ref{rem: real and imag confluent roots} belong to $\mathbf{Z}(f_m)$. We see that there are no other confluent roots as the complex band in (A) is isolated. Computation performed for $m = 3$, $\kappa = 3.5$, $\rho = 6.5$.}
    \label{Fig: Lambda spanned by polar curve}
\end{figure}

\begin{figure}[h]
    \centering
    \subfloat[][Complex band structure. The gap bands overlap, which is a clear hint that $\mathbf{Z}(f_m)$ is not spanned by a polar curve.]%
    {\includegraphics[width=0.48\linewidth]{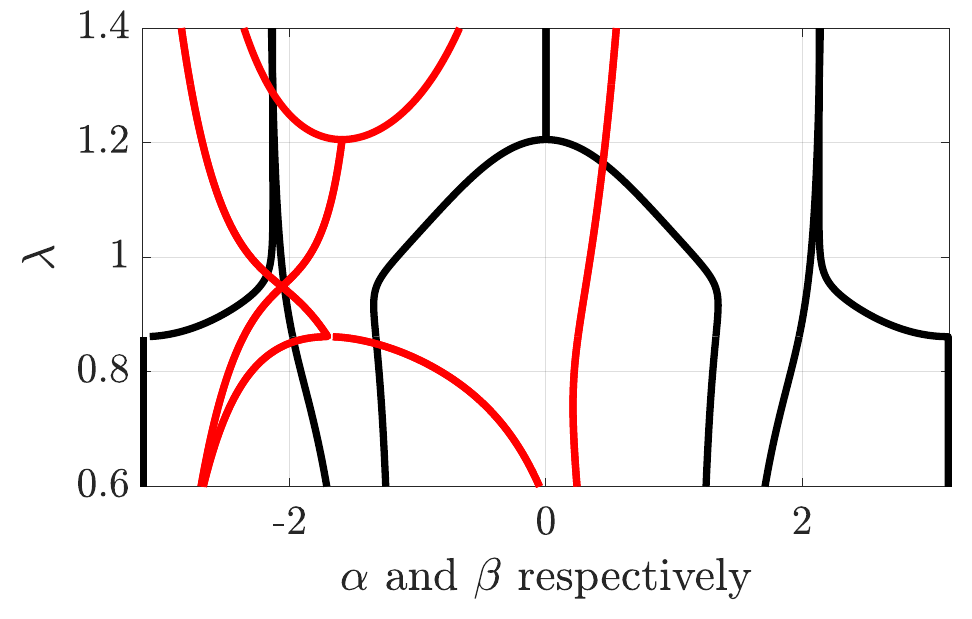}}\quad
    \subfloat[][The set $\mathbf{Z}(f_m)$ is not spanned by a polar curve.]
    {\includegraphics[width=0.47\linewidth]{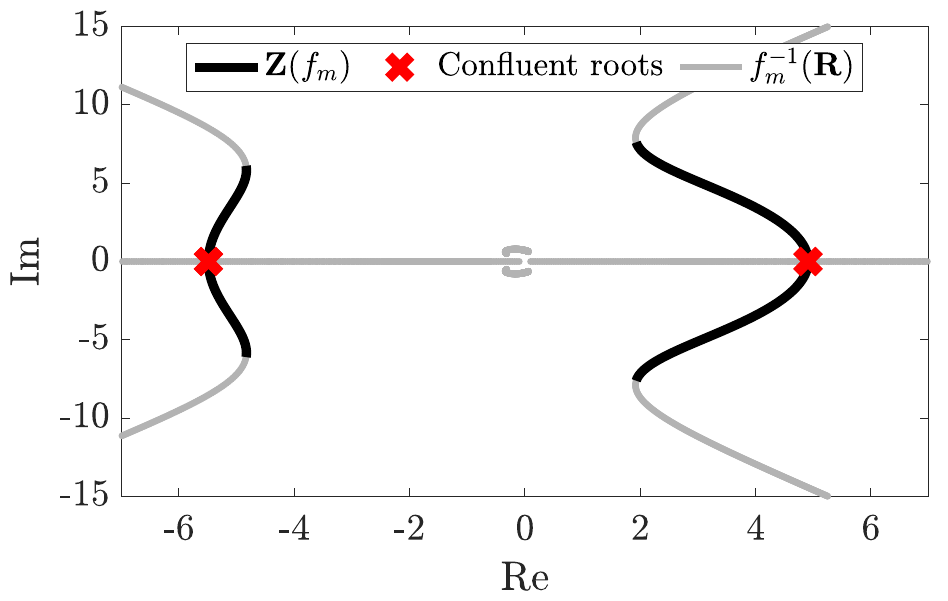}\vspace{1mm}}
    \caption{The set $\mathbf{Z}(f_m)\subseteq f_m^{-1}(\R)$ is no longer parametrised by a polar curve. Equivalently the gap bands in $(A)$ cross each other. The open limit in this setting contains complex values. Computation performed for $m = 3$, $\kappa =1$, $\rho = 6.5$.}
    \label{Fig: Lambda not spanned by polar curve}
\end{figure}

A visual criterion for when the open limit has real spectrum is if the gap band closest to the origin does not touch or overlap any other gap bands, as the corresponding real band sweeps over the first Brillouin zone.
It is worth mentioning that Theorem \ref{Thm: real sectra criterion} comes as no surprise, as the result has been conjectured and empirically observed under weaker assumptions, that is, if $f_m^{-1}(\R)$ contains a Jordan curve. This is the topic of study in \cite[Theorem 1]{RealSpectrumToeplitz}. However, the proof of this result was later discovered to be erroneous \cite{CorrectionRealAsymptoticSpectrum}, but despite the error, the result is commonly believed to be true, which is why we state it as a conjecture.

\begin{conjecture}\label{conjecture: real spectra}
    Let $f_m$ be the symbol function defined in \eqref{eq: def symbol function banded}, and let us denote by $f_m^{-1}(\R)$ the preimage of the symbol function, i.e. the roots of $f_m(z)-\lambda = 0$ for $\lambda\in \R$. The following statements are equivalent,
    \begin{enumerate}
        \item There exists a Jordan curve parametrised by the complex band structure.
        \item  $ \lim_{n\to\infty}\sigma\bigl(\mathbf{T}_n(f_m)\bigr) \subset\R.$
    \end{enumerate}
\end{conjecture}

Since only the implication $(1) \implies (2)$ was affected by the error \cite{CorrectionRealAsymptoticSpectrum}, the implication $(2) \implies (1)$ holds.

\subsection{Asymptotic spectra under homotopic deformations of the unit torus.}
As discussed in Theorem \ref{Thm: real sectra criterion}, the open limit is not affected by a specific deformation of the path along which the symbol function is evaluated. In fact, we will see that, the density of states is conserved under such a deformation.
The \emph{density of states} (DoS) of a matrix $\mathbf{A}_n$ is defined as the weak limit as $n\to \infty$ of the empirical measures,
\begin{equation}\label{eq: empirical measure}
    \mu_n(x) := \frac{1}{n}\sum_{k = 1}^n \delta_{\lambda_{k, n}}(x),
\end{equation}
where $\lambda_{1,n}, \dots, \lambda_{n,n}$ denote the eigenvalues of the matrix $\mathbf{A}_n$.
For non-Hermitian Toeplitz matrices, the density of states is characterised by the following result \cite{HirschmannSpectra}.

\begin{theorem}[Hirschman]\label{thm: Hirschmann weak limit}
    The family of empirical measures $(\mu_n)_{n \in \N}$ converges weakly to a measure that is supported on $\sigma_{\mathrm{{open}}}\bigl(\mathbf{T}(f)\bigr)$ and equals
    \begin{equation}\label{eq: measure Hirschmann}
        \d\mu = \frac{1}{2\pi}\frac{1}{g}\left|\frac{\partial g}{\partial n_1} + \frac{\partial g}{\partial n_2}\right| \d s,
    \end{equation}
    where $\d s$ is the arc length measure on $\sigma_{\mathrm{open}}\bigl(\mathbf{T}(f)\bigr)$. The function $g$ is defined as,
    \begin{equation}
        g(\lambda) := |a_m|\prod_{k=m+1}^{2m}|z_k(\lambda)|,
    \end{equation}
    and $z_i(\lambda)$ are the roots of the polynomial equation $z^m\bigl(f_m(z)-\lambda\bigr)=0$ sorted in ascending magnitude. 
\end{theorem}

We now turn to one of the central theorems of this paper, namely that the transformation of the Toeplitz matrix, by evaluating the symbol function on the polar curve $p(\mathbb{T})$, can be seen as an asymptotic similarity transform, which on average keeps the eigenvalues the same. This averaging of asymptotic spectra is precisely captured by the density of states. 

\begin{theorem}\label{thm: convergence of DoS}
     Suppose that $\mathbf{Z}(f_m)$ is spanned by a polar curve $p(\mathbb{T})$ that is $C^1$ defined in \eqref{eq: polar curve bijection} and suppose that Assumption \ref{Confluence Assumption} is met, then
    \begin{equation}
        \dos\bigl(\mathbf{T}_n(f_m) \bigr) = \dos\bigl(\mathbf{T}_n(f_m\circ p)\bigr).
    \end{equation}
\end{theorem}

\begin{proof}
Denote by $\mu_n$ and $\tilde\mu_n$ the empirical measures \eqref{eq: empirical measure} of $\mathbf T_n(f_m)$ and of $\mathbf T_n(f_m\circ p)$ respectively. Since $f_m$ has real coefficients, the curve $p(\mathbb{T})$ is invariant under complex conjugation and
\begin{equation}\label{eq: evenness}
    r(-\theta) = r(\theta) \qquad\text{and}\qquad \lambda(-\theta) = \lambda(\theta), \qquad \theta\in[-\pi,\pi].
\end{equation}
Moreover $\lambda=f_m\circ p$ is real valued and bounded, since $p$ is continuous and $f_m$ is continuous on the compact set $p(\mathbb{T})$. In particular $f_m,\;f_m\circ p\in L^{\infty}(\mathbb{T})$.

\medskip
\noindent\textit{Step 1: Trace asymptotics for both families.}
By the first Szeg\H{o} limit theorem in the form of \cite[Lemma 5.16, Theorem 5.17]{LargeTruncatedToeplitz}, for every $a\in L^{\infty}(\mathbb{T})$ and every $s\in\N$,
\begin{equation}\label{eq: BS trace formula}
    \tfrac1n\operatorname{tr}\bigl(\mathbf T_n(a)^{s}\bigr) = c_0\bigl(a^{s}\bigr)+o(1), \qquad n\to\infty .
\end{equation}
Applied to $a=f_m$ and to $a=f_m\circ p$ this gives
\begin{equation}\label{eq: two trace limits}
    \lim_{n\to\infty}\tfrac1n\operatorname{tr}\bigl(\mathbf T_n(f_m)^{s}\bigr) = c_0\bigl(f_m^{s}\bigr), \qquad \lim_{n\to\infty}\tfrac1n\operatorname{tr}\bigl(\mathbf T_n(f_m\circ p)^{s}\bigr) = c_0\bigl((f_m\circ p)^{s}\bigr).
\end{equation}

\medskip
\noindent\textit{Step 2: Deforming $\mathbb{T}$ into $p(\mathbb{T})$ leaves the
zeroth Fourier coefficient unchanged.}
Fix $s\in\N$ and put $g:=f_m^{s}$, a Laurent polynomial with real coefficients, say $g(z)=\sum_{|k|\le sm}b_k z^{-k}$, so that $c_0(g)=b_0$ and $g\circ p=\lambda^{s}$. Since $r\in C^{1}$ and $r>0$, the curve $p(\mathbb{T})$ is a closed $C^{1}$ curve in $\C\setminus\{0\}$, $z(\theta):=r(\theta)e^{\i\theta}$, and
\begin{equation}
    \frac{z'(\theta)}{z(\theta)} = \frac{r'(\theta)}{r(\theta)} + \i , \qquad\text{that is,}\qquad \i\,\d\theta = \frac{\d z}{z} - \d\bigl(\ln r(\theta)\bigr).
\end{equation}
Multiplying by $g(z(\theta))$ and integrating over $[-\pi,\pi]$ gives
\begin{equation}\label{eq: integral form}
    \i\int_{-\pi}^{\pi} g\bigl(z(\theta)\bigr)\,\d\theta \;=\; \oint_{p(\mathbb{T})}\frac{g(z)}{z}\,\d z \;-\; \int_{-\pi}^{\pi} g\bigl(z(\theta)\bigr) \frac{r'(\theta)}{r(\theta)}\,\d\theta .
\end{equation}
In the contour integral, expand $g(z)/z=\sum_{|k|\le sm}b_k z^{-k-1}$. For $k\neq 0$ the integrand $z^{-k-1}$ has the primitive $-z^{-k}/k$ on $\C\setminus\{0\}$, so its integral over the closed curve $p(\mathbb{T})$ vanishes, whereas for $k=0$, using $r(-\pi)=r(\pi)$,
\begin{equation}
    \oint_{p(\mathbb{T})}\frac{\d z}{z} = \int_{-\pi}^{\pi}\left(\frac{r'(\theta)}{r(\theta)}+\i\right)\d\theta = \Bigl[\ln r(\theta)\Bigr]_{-\pi}^{\pi} + 2\pi \i = 2\pi \i ,
\end{equation}
so that 
\begin{equation}
    \oint_{p(\mathbb{T})} g(z)z^{-1}\,\d z = 2\pi \i\,b_0 = 2\pi \i\,c_0(f_m^{s}).
\end{equation}
In the last term of \eqref{eq: integral form}, the factor $g(z(\theta))=\lambda(\theta)^{s}$ is even by \eqref{eq: evenness}, while $r'/r$ is odd because $r$ is even. The integrand is therefore odd on the symmetric interval $[-\pi,\pi]$ and the integral vanishes. Hence \eqref{eq: integral form} reduces to
\begin{equation}\label{eq: zeroth coefficients agree}
    c_0\bigl((f_m\circ p)^{s}\bigr) = \frac{1}{2\pi}\int_{-\pi}^{\pi}\lambda(\theta)^{s}\,\d\theta = c_0\bigl(f_m^{s}\bigr), \quad \forall s\in\N .
\end{equation}

\medskip
\noindent\textit{Step 3: Convergence of moments.}
By Theorem \ref{Thm: real sectra criterion},
\begin{equation}\label{eq: open limits agree}
    \sigma_{\mathrm{open}}\bigl(\mathbf T(f_m)\bigr) = \sigma_{\mathrm{open}}\bigl(\mathbf T(f_m\circ p)\bigr) = [A,B] \subset \R ,
\end{equation}
and this set is compact and connected by Theorem~\ref{thm: connectedness of open limit}. By Theorem~\ref{thm: Hirschmann weak limit}, $\mu_n\rightharpoonup\mu$ for a probability measure $\mu$ supported on $[A,B]$. By Proposition \ref{prop: eval convex humm} every $\mu_n$ is supported in the fixed compact set $K:=\mathrm{conv}\,\mathcal R(f_m)\subset\C$, so, choosing $\varphi\in C_c(\C)$ with $\varphi(z)=z^{s}$ on $K\cup[A,B]$, weak convergence together with $\operatorname{tr}(A^{s})=\sum_j\lambda_j(A)^{s}$ (algebraic multiplicities counted) and \eqref{eq: two trace limits} yields
\begin{equation}\label{eq: moments of mu}
    \int_{\R}\lambda^{s}\,\d\mu(\lambda) = \lim_{n\to\infty}\tfrac1n\operatorname{tr}\bigl(\mathbf T_n(f_m)^{s}\bigr) = c_0\bigl(f_m^{s}\bigr), \quad \forall s\in\N .
\end{equation}

By Theorem \ref{Thm: real sectra criterion}, every $\tilde\mu_n$ is a probability measure on the fixed compact interval $I:=[A,B]$, and by
\eqref{eq: two trace limits}
\begin{equation}\label{eq: moments of mutilde n}
    \int_{\R} x^{s}\,\d\tilde\mu_n(x) = \tfrac1n\operatorname{tr}\bigl(\mathbf T_n(f_m\circ p)^{s}\bigr)  \;\xrightarrow[n\to\infty]{}\; c_0\bigl((f_m\circ p)^{s}\bigr),  \quad \forall s\in\N .
\end{equation}
Since all $\tilde\mu_n$ live on $I$, the sequence is tight, and any two subsequential weak limits are probability measures on $I$ with the same moments \eqref{eq: moments of mutilde n}. As polynomials are dense in $C(I)$ by the Weierstrass approximation theorem, such limits coincide. Hence $\tilde\mu_n\rightharpoonup\tilde\mu$, where $\tilde\mu$ is the unique probability measure on $I$ with moments $c_0\bigl((f_m\circ p)^{s}\bigr)$. The pushforward of the normalised Lebesgue measure on $\mathbb{T}$ under $\lambda$ has exactly these moments, namely $\frac{1}{2\pi}\int_{-\pi}^{\pi}\lambda(\theta)^{s}\,\d\theta = c_0\bigl((f_m\circ p)^{s}\bigr)$, and is supported on $\lambda([-\pi,\pi])=[A,B]$ by \eqref{eq: spectral limit}. By uniqueness,
\begin{equation}\label{eq: hermitian DoS}
    \int_{\R}\varphi\,\d\tilde\mu = \frac{1}{2\pi}\int_{-\pi}^{\pi}\varphi\bigl(\lambda(\theta)\bigr)\,\d\theta, \quad \varphi\in C(\R),
\end{equation}
and $\operatorname{supp}\tilde\mu\subset[A,B]$.

\medskip
\noindent\textit{Step 4: Conclusion.}
Combining \eqref{eq: moments of mu}, \eqref{eq: zeroth coefficients agree} and \eqref{eq: moments of mutilde n},
\begin{equation}
    \int_{\R}\lambda^{s}\,\d\mu = c_0\bigl(f_m^{s}\bigr) = c_0\bigl((f_m\circ p)^{s}\bigr) = \int_{\R}\lambda^{s}\,\d\tilde\mu, \quad \forall s\in\N .
\end{equation}
Both $\mu$ and $\tilde\mu$ are probability measures supported on the compact interval $[A,B]\subset\R$, on which polynomials are dense in $C([A,B])$. Hence $\int\varphi\,\d\mu=\int\varphi\,\d\tilde\mu$ for all $\varphi\in C([A,B])$, and $\mu=\tilde\mu$ as Borel measures on $\R$ by the Riesz representation theorem. Together with \eqref{eq: hermitian DoS} this proves the claim.
\end{proof}

In Figure \ref{Fig: QuasiSimilarity} we numerically illustrate the convergence of the empirical measure for a non-Hermitian banded Toeplitz matrix, for which there exists a polar curve along which the symbol function is real-valued.

\begin{figure}[h]
    \centering
    \subfloat[][Polar curve $p(\mathbb{T})$ and the unit torus $\mathbb{T}$.]%
    {\includegraphics[width=0.3\linewidth]{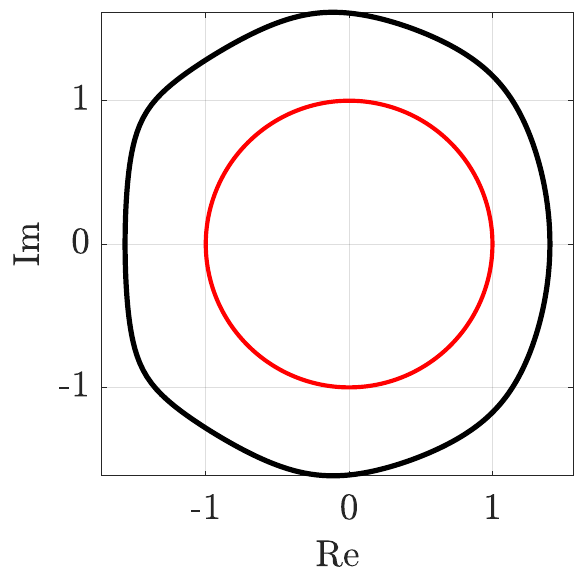}}\quad
    \subfloat[][Density of states of  $\mathbf{T}_n(f)$ and the empirical measure of $\mathbf{T}_n(f\circ p)$.]%
    {\includegraphics[width=0.65\linewidth]{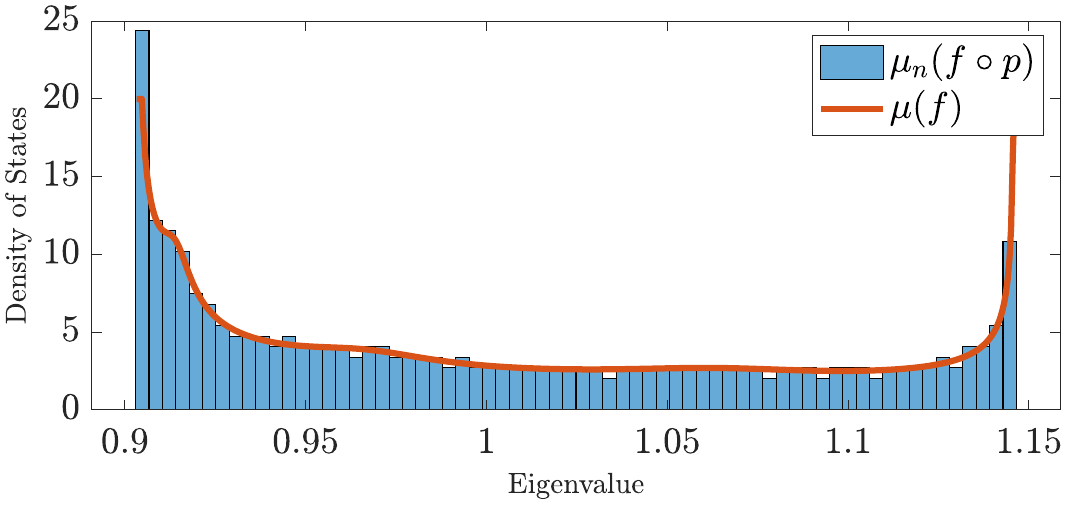}}
    \caption{Figure (A) illustrates the polar curve, which is not just a scaling of the torus but rather a homotopic deformation of the torus. Figure (B) illustrates that even for relatively small matrix sizes, i.e. $n = 100$, the empirical measure of $\mathbf{T}_n(f\circ p)$ approaches the density of states $\mathbf{T}_n(f)$. Computation performed for $\kappa = 3.5 $, $\rho = 4.8$ and $m = 7$.}
    \label{Fig: QuasiSimilarity}
\end{figure}  

In the proof of Theorem \ref{thm: convergence of DoS} a formula for the DoS of Hermitian Toeplitz matrices was presented. Note that this formula is considerably simpler than the general formula for non-Hermitian Toeplitz matrices presented in \eqref{eq: measure Hirschmann}. In other words, to find the density of states for a non-Hermitian matrix $\mathbf{T}_n(f)$, it is sufficient to compute the density of states for the Hermitian matrix $\mathbf{T}_n(f\circ p)$. We have the following result.

\begin{corollary}
    Suppose that the contour traced out by \eqref{eq: polar curve bijection} is $C^1([0, 2\pi])$, then for a Borel subset $E\subset \R$ the $\dos\bigl(\mathbf{T}_n(f) \bigr)$ is given by the measure,
    \begin{equation}\label{eq: DOS Hermitian}        \mu(E) = \frac{1}{2\pi}\int_{0}^{2\pi} \chi_E\big(f\circ p(e^{\i\theta})\big) \d\theta.
    \end{equation}
\end{corollary}

\begin{proof}
    As the assumptions for Theorem \ref{thm: convergence of DoS} are met, computing the DoS of $\mathbf{T}_n(f)$ is equivalent to computing the DoS of the Hermitian matrix $\mathbf{T}_n(f\circ p)$.
    Following the formula of the density of states of a Hermitian Toeplitz matrix \cite[Section 7.5]{GrenanderSzego+1958}, it follows that the density of states is given by \eqref{eq: DOS Hermitian}.
\end{proof}

\begin{remark}
    The substitution $f_m\mapsto f_m\circ p$ preserves exactly one piece of spectral information, and it is worth recording precisely which. Equation \eqref{eq: zeroth coefficients agree} states that
    \begin{equation}\label{eq:firstorder}
        c_0\bigl((f_m\circ p)^s\bigr)=c_0\bigl(f_m^s\bigr),\quad s\in\mathbb N,
    \end{equation}
    equivalently, by the first Szeg\H{o} limit theorem,
    \begin{equation}
        \lim_{n\to\infty}\tfrac1n\operatorname{tr} \mathbf T_n(f_m)^s =\lim_{n\to\infty}\tfrac1n\operatorname{tr} \mathbf T_n(f_m\circ p)^s,\quad s\in\mathbb N .
    \end{equation}
    That is, the deformation leaves the \emph{first-order} Szeg\H{o} data invariant. In order to achieve sharper pointwise eigenvalue estimates, the higher order  Szeg\H{o} terms would have to be preserved, which is only possible for $p: z \mapsto rz,~r>0$. This will be investigated further in a forthcoming paper.
\end{remark}

In conclusion, the class of non-Hermitian Toeplitz matrices, as intended by Theorem \ref{thm: convergence of DoS}, many results such as eigenvalues estimates and convergence rates follow from existing results for Hermitian matrices. We refer the reader to \cite{BookSpectraBandedToepllitz,ComputeBandedToeplitz, doi:10.1137/22M1529920, LargeTruncatedToeplitz} for a detailed study of the subject.

\subsection{Density of states of asymptotically equivalent matrices.} In the present section, we demonstrate that the convergence of density of states presented in Theorem \ref{thm: convergence of DoS} may also be applied to matrices which are not entry-wise Toeplitz. This result will be of particular importance for considering finite but large physical systems, as the finiteness often incorporates possible edge effects, which may perturb the edge entries of the  Toeplitz matrix approximation. 
We will present results similar to \cite{ammari2023spectralConvergence} for non-Hermitian matrices with real asymptotic spectra. For an $n\times n$ matrix $\mathbf{M}$, let us define the normalised Frobenius norm as
\begin{equation}
    \lvert \mathbf{M} \rvert^2 := \frac{1}{n} \sum_{i,j}^n |\mathbf{M}_{i,j}\rvert^2.
\end{equation}

\begin{definition}
    Two sequences of $n\times n$ matrices $\mathbf{A}_n$ and $\mathbf{B}_n$ are said to be \emph{asymptotically equivalent} if they have bounded operator norm, that is, $\lVert \mathbf{A}_n \rVert$, $\lVert \mathbf{B}_n \rVert < \infty$, and 
    \begin{equation}
        \lim_{n \to \infty} \lvert \mathbf{A}_n - \mathbf{B}_n\rvert = 0.
    \end{equation}
    If one of the two matrices is a Toeplitz matrix, then the other is said to be \emph{asymptotically Toeplitz}.
\end{definition}
We have the following result for  asymptotically equivalent matrices \cite[Corollary 2.1]{1054924}.

\begin{theorem}
    Let $\mathbf{A}_n$ and $\mathbf{B}_n$ be asymptotically equivalent sequences of matrices with eigenvalues $a_{n,k}$ and $b_{n,k}$ respectively, then the moments agree, that is, 
    \begin{equation}
        \lim_{n\to \infty}\frac{1}{n} \sum_{k = 1}^n a_{n,k}^s = \lim_{n\to \infty}\frac{1}{n} \sum_{k = 1}^n b_{n,k}^s, ~\forall s \in \N.
    \end{equation}
\end{theorem}

Asymptotic equivalence is not sufficient to specify the pointwise behaviour of individual eigenvalues, but it allows one to characterise their density of states. 
We have the following result on the spectral distribution of asymptotically equivalent matrices.
\begin{corollary}\label{cor: Asymptotic equivalence and DoS}
    Let $\mathbf{A}_n$ be a matrix that has asymptotically real spectra and is asymptotically equivalent to a Toeplitz matrix $\mathbf{T}_n$ such that $\sigma_{\mathrm{open}}\bigl(\mathbf{T}(f)\bigr) \subseteq \R$, then both matrices asymptotically have the same density of states.
\end{corollary}

\begin{proof}
    We have that all moments are finite,
    \begin{equation}
        \lim_{n\to\infty} \int x^s \d \mu_n^{(\mathbf{A}_n)}(x) = \lim_{n\to \infty}\frac{1}{n} \sum_{k = 1}^n a_{n,k}^s = \lim_{n\to \infty}\frac{1}{n} \sum_{k = 1}^n t_{n,k}^s = \lim_{n\to\infty} \int x^s \d \mu_n^{(\mathbf{T}_n)}(x), ~\forall s \in \N.
    \end{equation}
    Moreover, by Theorem \ref{thm: Hirschmann weak limit} $\mu_n^{(\mathbf{T}_n)}\rightharpoonup\mu^{(\mathbf{T}_n)}$ and since the measure $\mu$ is supported on a compact subset of $\R$, the moments uniquely determine the measure, from which it follows that $\mu^{(\mathbf{T}_n)} = \mu^{(\mathbf{A}_n)}$, which completes the proof.
\end{proof}

In other words, if a matrix $\mathbf{A}_n$ is asymptotically equivalent to a Toeplitz matrix $\mathbf{T}_n$ then their eigenvalues follow the same limiting distribution. This result will be of particular interest when considering the limit of finite but large systems, where the Toeplitz matrix representation may have perturbed entries, due to possible edge effects \cite{ammari2023nonhermitianskineffectthreedimensional, debruijn2025spectrapseudospectranonhermitiantoeplitz}

\subsection{Asymptotic similarity transform}\label{Sec: Asymptotic Similarity}
A common problem that is often encountered when computing the spectra of large non-Hermitian Toeplitz matrices is numerical pollution, as illustrated in Figure \ref{Fig: Eigval conditionner}. This comes from the fact that non-Hermitian Toeplitz matrices are often severely ill-conditioned \cite{trefethen.embree2005Spectra, REICHEL1992153, ComputeBandedToeplitz}. This computational bottleneck will be addressed in this section by introducing an asymptotic similarity transform.
In order to visualise the spectral pollution, we recall the following result \cite[Proposition 2.17.]{LargeTruncatedToeplitz}
\begin{proposition}\label{prop: eval convex humm}
    If $f \in L^\infty$, then $\sigma\bigl(\mathbf{T}_n(f)\bigr) \subset \operatorname{conv} \mathcal{R}(f)$, where $\mathcal{R}(f)$ is the essential range of $f$.
\end{proposition} 
Consequently, because of the similarity of $\mathbf{T}_n\big(f(\mathbb{T})\big) \sim \mathbf{T}_n\big(f(r\mathbb{T})\big)$ choosing $r$ in such a manner that $\operatorname{conv} \mathcal{R}\big(f(r\mathbb{T})\big)$ has the smallest measure, corresponds to preconditioning the matrix $\mathbf{T}_n\big(f(\mathbb{T})\big)$, see \cite{Beam1993TheAS}.

\begin{figure}[h]
    \centering
    \subfloat[][The numerically computed eigenvalues of $\mathbf{T}_n(f\circ p)$ and $\mathbf{T}_n(f)$ for $n = 40$ agree closely. A more detailed convergence study will be presented in Figure \ref{Fig: Pointwise eigenvalue convergence}.]%
    {\includegraphics[width=0.48\linewidth]{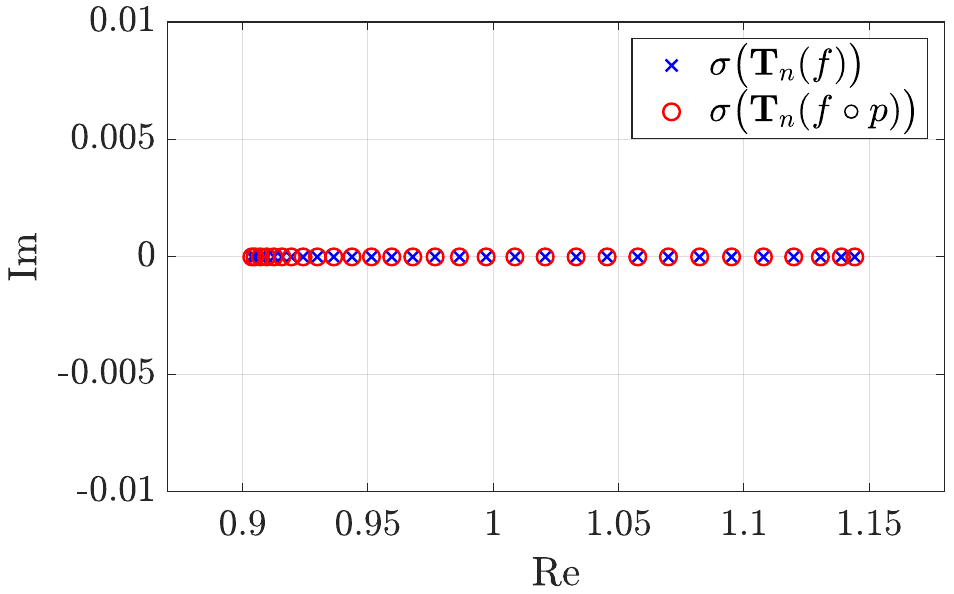}}\quad
    \subfloat[][Numerically computed spectra of $\mathbf{T}_n(f\circ p)$ and $\mathbf{T}_n(f)$ for $n = 2000.$ Many of the eigenvalues of $\mathbf{T}_n(f)$ fall outside of $\operatorname{conv}\mathcal{R}\big(f(r\mathbb{T})\big)$ and by Proposition \ref{prop: eval convex humm} constitute numerical error.]%
    {\includegraphics[width=0.49\linewidth]{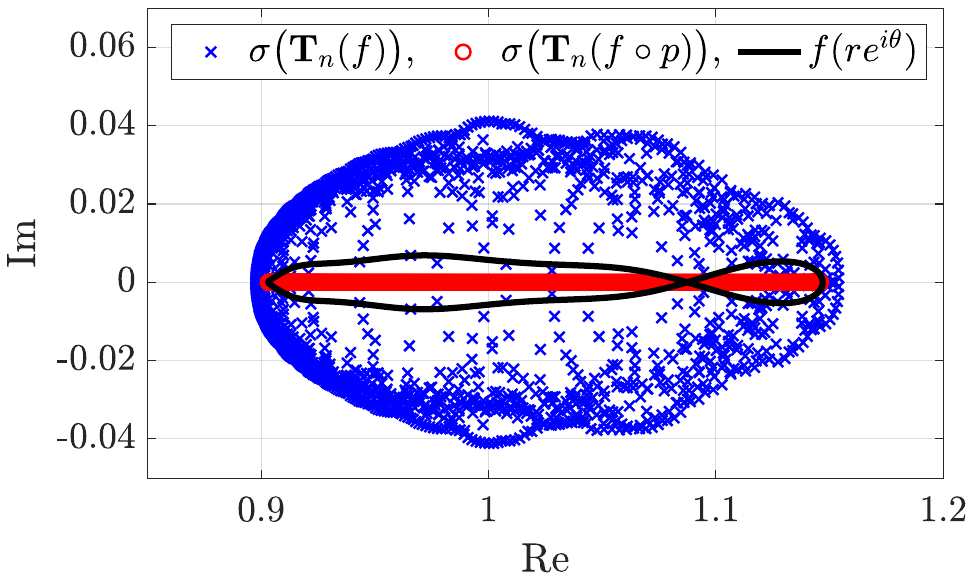}}
    \caption{The Hermitian matrix $\mathbf{T}_n(f\circ p)$ is less prone to  numerical pollution compared to the non-Hermitian matrix $\mathbf{T}_n(f)$. This motivates the idea of considering the evaluation of the symbol function on a deformed torus as a numerical preconditioner. Computation performed for $\kappa = 3.5 $, $\rho = 4.8$ and $m = 7$.}
    \label{Fig: Eigval conditionner}
\end{figure} 

In the case of tridiagonal matrices, the problem may be overcome by shrinking $\operatorname{conv}\mathcal{R}\big(f(r\mathbb{T})\big)$ to a line, just by scaling the unit torus with a correctly chosen $r$, see \cite{ammari2024spectra, ammari2024generalisedbrillouinzonenonreciprocal}. For banded Toeplitz matrices, this collapsing of the unit circle no longer applies as in the tridiagonal case, which is why we will allow general polar curves $p:\mathbb{T} \to \C$ instead of just scaling of the unit torus $p(z) = rz$. In this manner, we show that under the same assumptions as in Theorem \ref{Thm: real sectra criterion}, $\operatorname{conv}\mathcal{R}\big(f\circ p(\mathbb{T})\big)$ will collapse onto the line.

For non-Hermitian banded Laurent operator, i.e. the doubly infinite Toeplitz operator, one may formally consider the asymptotic similarity transform $\mathbf{L}(f) = \mathbf{M}^{-1}\mathbf{L}(f\circ p)\mathbf{M}$, where

\begin{equation}\label{eq: similarity transfrom}
    (\mathbf{M})_{i, j} = \frac{1}{2\pi}\int_0^{2\pi} p(e^{\i \theta})^{-j} e^{\i\theta i} \d \theta, \quad i,j\in \mathbb{Z}.
\end{equation}
In order to achieve a similar transform for finite Toeplitz matrices the Laurent operator would have to be truncated accordingly. 

We would like to emphasise that the asymptotic similarity transform presented in \eqref{eq: similarity transfrom} is the natural generalisation of the exact similarity transform given by a scaling of the torus. For a function $p(z) = rz$ for $r > 0$, the transform in \eqref{eq: similarity transfrom} reduces to the known exact similarity transform $\mathbf{M} = \operatorname{diag}(1, r, r^2, \dots, r^{n-1})$. This is due to the fact that $\mathbf{M}$ commutes with the projection $\mathbf{P}_n$.
We emphasise that the similarity transform introduced in \eqref{eq: similarity transfrom} is exponentially ill-conditioned, which is why we never use it in numerical computations.  This is due to the fact that the entries grow or decay exponentially in the index $j$. In practical terms, we instead compute the symmetrised matrix $\mathbf{T}_n(f\circ p)$ by numerically evaluating the Fourier coefficients $(\hat{f\circ p})_k$. More details on the computation of the Fourier coefficients can be found in the Appendix \ref{sec: Numerical analysis}. 

The transform $\mathbf{T}_n(f) \to \mathbf{T}_n(f\circ p)$ is expected to not be an exact similarity transform for finite $n$. This is due to the fact that the Fourier coefficients $(\hat{f\circ p})_k$ are, in general, no longer banded; see, for example, Figure \ref{Fig: Fourier coefficents composition} (B). Consequently, the coefficients $\big(\hat{f\circ p}\big)_k$ for $|k| >n$ will not influence the matrix $\mathbf{T}_n(f\circ p)$. However, we will provide numerical evidence that the spectrum is conserved pointwise under the transform $\mathbf{T}_n(f) \to \mathbf{T}_n(f\circ p)$ in Figure \ref{Fig: Pointwise eigenvalue convergence}.
For more details on the derivation of the similarity transform, we refer the reader to the Appendix \ref{Appendix: Similarity transform}.

We will consider the error of the spectrum when introducing the transform defined by $p$. Let us denote by $\lambda_k$ and $\tau_k$ the sorted eigenvalues of $\mathbf{T}_n(f)$ and $\mathbf{T}_n(f\circ p)$, respectively, then we define the $\ell^1$ distance between the spectra as:
\begin{equation}\label{def: l1 distance}
    d_\sigma\big(\mathbf{T}_n(f), \mathbf{T}_n(f\circ p)\big) := \sum_{k =1 }^n |\lambda_k-\tau_k|.
\end{equation}
Note that the $\ell^1$ distance is very sensitive to finitely many outliers.
Under the assumptions of Theorem \ref{Thm: real sectra criterion} the eigenvalues of $\mathbf{T}_n(f)$ and $ \mathbf{T}_n(f\circ p)$ will cluster with the same distribution along the same limiting set which consists of a finite interval on the real axis. As a consequence, it is expected that $|\lambda_k-\tau_k| = \mathcal{O}(n^{-1})$. This is the reason why the $\ell^1$ distance defined in \eqref{def: l1 distance} does not have a normalisation of $1/n$. An illustration of this is given in Remark \ref{rem: same open limit different spectra}.

\begin{remark}\label{rem: same open limit different spectra}
We would like to highlight that the observed  convergence rate of the spectrum of $\mathbf{T}_n(f)$ and $\mathbf{T}_n(f\circ p)$ as $n$ grows is not simply a consequence of the fact that the limiting spectra agree. For comparison, we consider the symbol functions $f(\theta) = \cos(\theta)$ and $g(\theta) = \cos(2 \theta)$,   which both correspond to Hermitian matrices. It is not hard to see that their open spectra agree: $\sigma_{\mathrm{open}}(f) = \sigma_{\mathrm{open}}(g) = [-1, 1]$ even though their DoS do not agree. The $\ell^1$ distance of this example is illustrated in Figure \ref{Fig: Pointwise eigenvalue convergence} (A), as it shows that the convergence observed in Figure \ref{Fig: Pointwise eigenvalue convergence} (B) is not a result of progressively filling the open limit with additional eigenvalues, but rather stems from sampling from the same distribution.
\end{remark}

\begin{figure}[h]
    \centering
    \subfloat[][Computation performed for $f(\theta) = \cos(\theta)$ and $g(\theta) = \cos(2\theta)$. The $\ell^1$ distance diverges even though their open limit agree.]%
    {\includegraphics[width=0.48\linewidth]{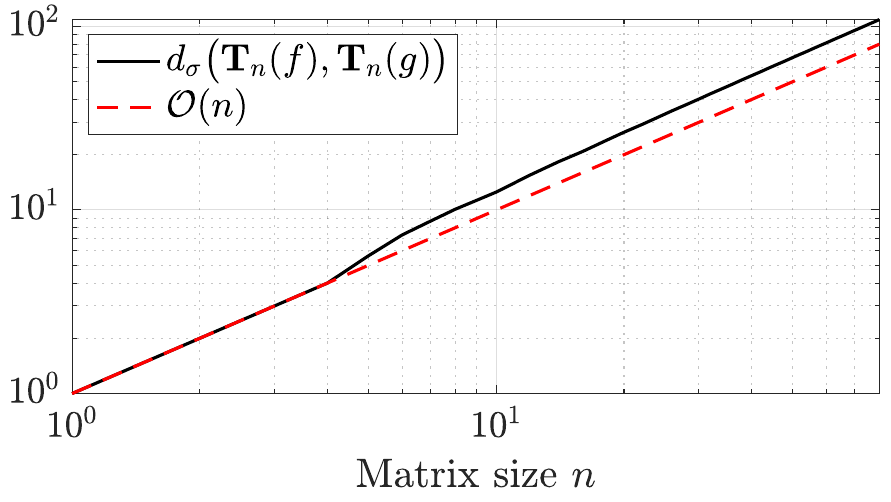}}\quad
    \subfloat[][Computation performed for $m=7, \kappa = 3.5$ and $\rho = 4.8$, i.e. the same Toeplitz matrix as in Figure \ref{Fig: Eigval conditionner} and Figure \ref{Fig: QuasiSimilarity}. ]%
    {\includegraphics[width=0.48\linewidth]{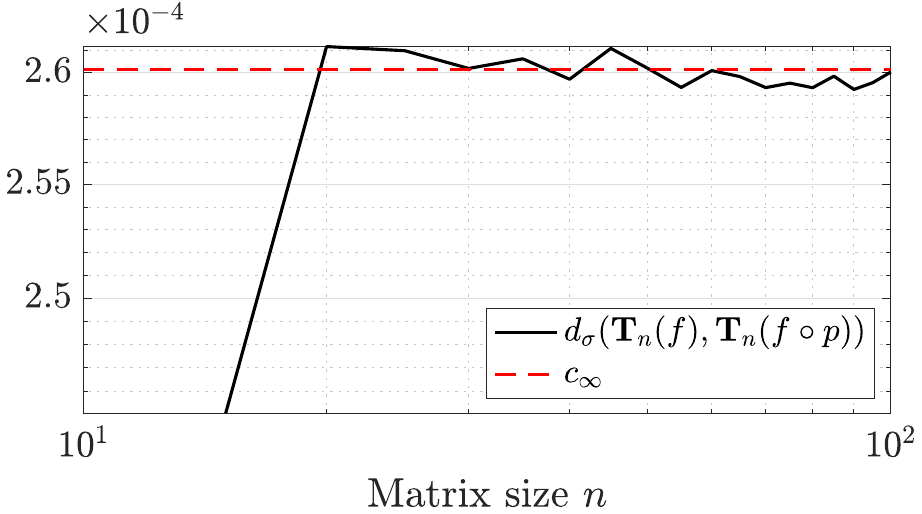}}
    \caption{Numerical experiments show that the $\ell^1$ distance converges to a constant, for Toeplitz matrices that have a closed polar curve $p(\mathbb{T})$. For more details on this we refer the reader to  Appendix \ref{sec: Numerical analysis}. Note that it is numerically difficult to simulate matrices for sizes larger than $n = 100$, see Figure \ref{Fig: Eigval conditionner}.}
    \label{Fig: Pointwise eigenvalue convergence}
\end{figure} 

In conclusion, in Theorem \ref{thm: convergence of DoS} we demonstrated that the density of states of $\mathbf{T}_n(f)$ and $\mathbf{T}_n(f\circ p)$ agree. However, this is not enough to guarantee pointwise convergence of the asymptotic spectrum, as there might be $o(n)$ many outliers without affecting the limiting distribution. Despite this, Figure \ref{Fig: Pointwise eigenvalue convergence} provides numerical corroboration that the transform $\mathbf{T}_n(f) \to \mathbf{T}_n(f\circ p)$ can be seen as an asymptotic similarity transform.

\section{Non-Hermitian Bloch Hamiltonians: Open and Periodic boundary Conditions}\label{sec: Bloch Hamiltonian}

Toeplitz matrix models find applications especially in the theory of condensed matter, for example, in tight-binding electron models and open crystal structures \cite{STEY1973213, Marques_2020}. In the limit as the matrix size tends to infinity, the energy spectrum becomes a piecewise continuous function, which we shall refer to as the open limit.
Non-Hermitian Hamiltonians are generally sensitive to boundary conditions, and the asymptotic spectra under open boundary conditions generally do not agree with the Bloch spectrum predicted for periodic boundary conditions \cite{Yang_2020, guo2021analysisbulkboundarycorrespondence, Ji_2024, PhysRevB.105.045422}.
However, using the insights gained in Section \ref{sec: Toeplitz Analysis}, in particular Theorem \ref{Thm: real sectra criterion} will be used to provide a mathematical foundation for the generalised Brillouin zone, capable of capturing the open spectra of non-Hermitian Hamiltonians \cite{PhysRevLett.123.066404}.
In the generalised Brillouin zone, the Bloch wave number is complex valued \cite{debruijn2024complexbandstructuresubwavelength}.
In systems involving Bloch Hamiltonians, the complex valued Bloch parameter is often called an \emph{imaginary gauge transformation}, relating non-reciprocal systems to equivalent Hermitian systems via a similarity transformation. Moreover, \Cref{Thm: real sectra criterion} gives a visually striking condition: if the generalised Brillouin zone is a $C^1$ polar curve, then the energy spectrum under open boundary conditions will be real.

For non-Hermitian Hamiltonians, three limits are considered. Periodic boundary conditions (PBC), semi-infinite boundary conditions (SIBC) and open boundary conditions (OBC). In \cite{HU202551, Okuma2020}  an imaginary gauge transformation on the Bloch Hamiltonian is considered, 
\begin{equation}\label{eq: gauoge transform}
    H(e^{\i\alpha}) \mapsto H\big(e^{\i(\alpha + \i r)}\big) =: H_r.
\end{equation}
Note that the imaginary gauge potential $r$ is independent of the Bloch parameter $\alpha$ and will therefore be denoted as \emph{constant gauge transform}.
This transformation does not alter the OBC spectra because it is implemented as a similarity transformation at the matrix level. 
By \Cref{thm: schmidt spitzer magnitude}, the spectrum of the open limit of a non-Hermitian resonator chain is given by the intersection of the spectra of the Hamiltonian under all constant imaginary gauge transformations \cite{NonHermitianTopoloficalReview}, that is
\begin{equation}\label{OBC infinite intersection}
    \sigma_{\mathrm{OBC}}(H) = \bigcap_{e^{-r} \in (0, \infty)} \sigma_{\mathrm{SIBC}}(H_r).
\end{equation}
The downside of this characterisation is that it is in general hard to conceptualise infinite intersections. We may evade taking infinite intersections in \eqref{OBC infinite intersection} provided that the assumptions of Theorem \ref{Thm: real sectra criterion} are met. In this case, the open limit is captured by evaluating the Bloch Hamiltonian along the  curve specified by $p:\mathbb{T}\to\C$, such that
\begin{equation}\label{eq:gbz}
     \sigma_{\mathrm{OBC}}(H) = \bigcup_{\alpha \in [-\pi,\pi]} \sigma\Big( H\big(p(e^{\i\alpha})\big) \Big). 
\end{equation}
The key assumption of \Cref{Thm: real sectra criterion} is that the generalised Brillouin zone is spanned by a polar curve $p(\mathbb{T})$. This assumption naturally appears in \eqref{eq:gbz}: the imaginary gauge transform $e^{\i\alpha} \to p(e^{\i\alpha})$ is uniquely specified by the Bloch parameter $\alpha$. 

Depending on the relative position of the generalised Brillouin zone with respect to the unit torus, we may observe left  or right eigenmode condensation. When the generalised Brillouin zone crosses the unit circle, bulk eigenmodes become partially left-localised and partially right-localised \cite{PhysRevB.109.035119, PhysRevB.110.205429}. This is numerically illustrated in Figure \ref{Fig: Hamiltonian crossing}.

\begin{figure}[h]
    \centering
    \subfloat[][Generalised Brillouin zone (solid line) and the unit torus (dashed).]%
    {\includegraphics[height=0.3\linewidth]{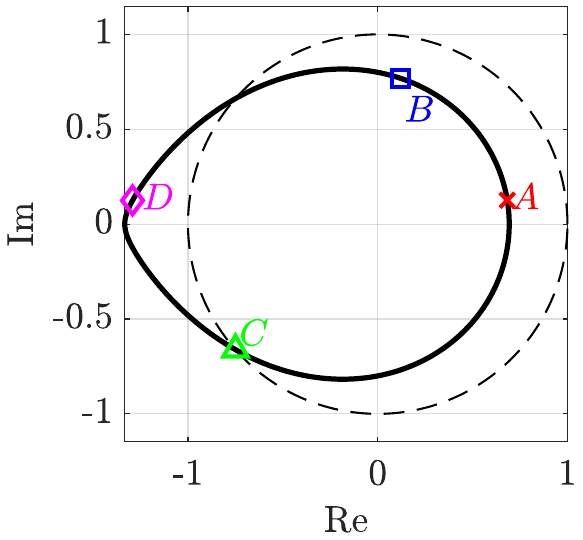}}\quad
    \subfloat[][Eigenvectors associated to the specific quasimomenta.]%
    {\includegraphics[height=0.31\linewidth]{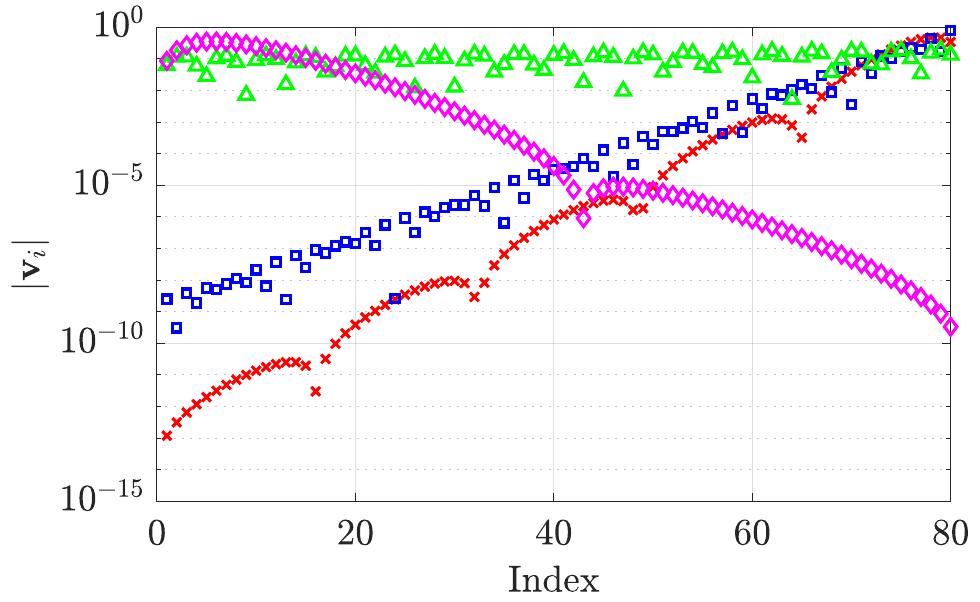}}
    \caption{Generalised Brillouin zone for a long-range non-Hermitian Hamiltonian $H(z) = 0.6z^{-2} + 2.5z^{-1} + 3.2 + 1.6z$. As the Brillouin zone is traced out by a $C^1$ polar curve, by Theorem \ref{Thm: real sectra criterion}, the open spectrum will be real. The decay rate of the eigenmodes associated to certain quasimomenta is correctly predicted by the generalised Brillouin zone.}
    \label{Fig: Hamiltonian crossing}
\end{figure}

Moreover, it has been observed physically that cusps in the generalized Brillouin zone (GBZ) correspond to exceptional points, at which the spectrum in the open limit may admit complex branches \cite{PhysRevResearch.2.043045, PhysRevLett.132.050402, PhysRevLett.123.066404}. Although non-Bloch PT symmetry breaking occurs through the formation of such cusps, the spectrum may still remain real precisely at the cusp itself. For this reason, we exclude cusps altogether in Assumption~\ref{Confluence Assumption}.

\section{Concluding remarks}\label{sec: Conclusion}

In this work, we analysed the asymptotic spectra of banded non-Hermitian Toeplitz matrices and established a rigorous mathematical foundation for the generalised Brillouin zone in non-Hermitian systems. The symmetrisation introduced in Section~\ref{Sec: Asymptotic Similarity} also suggests a practical route to computing the spectra of large banded non-Hermitian Toeplitz matrices, which are notoriously ill-conditioned.
A dedicated numerical study of this preconditioning strategy will be conducted in forthcoming work, together with a higher order expansion of the symmetrised symbol, which renders the $\ell^1$ spectral distance $d_\sigma$ convergent at order $\mathcal{O}(n^{-1})$. Furthermore, we plan to illustrate these results in a concrete PDE setting, showing that for periodic systems of non-Hermitian partial differential equations, the complex quasimomentum must also depend on the Floquet parameter. It would also be interesting to explore applications to specific models, such as the long-range Hatano-Nelson model or the subwavelength non-Hermitian skin effect.

\section{Acknowledgement}
The authors thank Clemens Thalhammer for pointing out an error in an earlier version of Theorem \ref{thm: convergence of DoS}.

\section{Conflict of interest}
The authors have no competing interests to declare.

\section{Data availability} \label{Sec: Data availability}
The \texttt{Matlab} code for the numerical experiments developed in this work is openly available in the following repository:
\url{https://github.com/yannick2305/OpenLimitToeplitz}.

\printbibliography

\appendix

\section{Confluent roots: Obstruction in the proof strategy}\label{appendix: Counterexample Confluent roots}

In this section, we examine where the proof strategy for Theorem \ref{Thm: real sectra criterion} may fail when the assumptions of distinct roots is not satisfied. In this setting the technique in the proof fails as the path $\tau$ to the real axis might be chosen in such a way that the roots become confluent on the real axis. This could lead to a possible bifurcation into the complex plane. Though we have not yet observed this for some symbol function, as effectively the parameters are poised on a set of measure zero, we will present a possible obstruction in Figure \ref{Fig: counterexample confluent roots}. We recall that following Conjecture \ref{conjecture: real spectra} the result is commonly believed to hold even in the general case \cite{CorrectionRealAsymptoticSpectrum}, but as of today we can not prove the most general setting yet.

\begin{figure}[h]
    \centering
    \subfloat[][Complex band structure with touching gap bands, i.e. the root modulus in no longer strictly separated.]%
    {\includegraphics[width=0.48\linewidth]{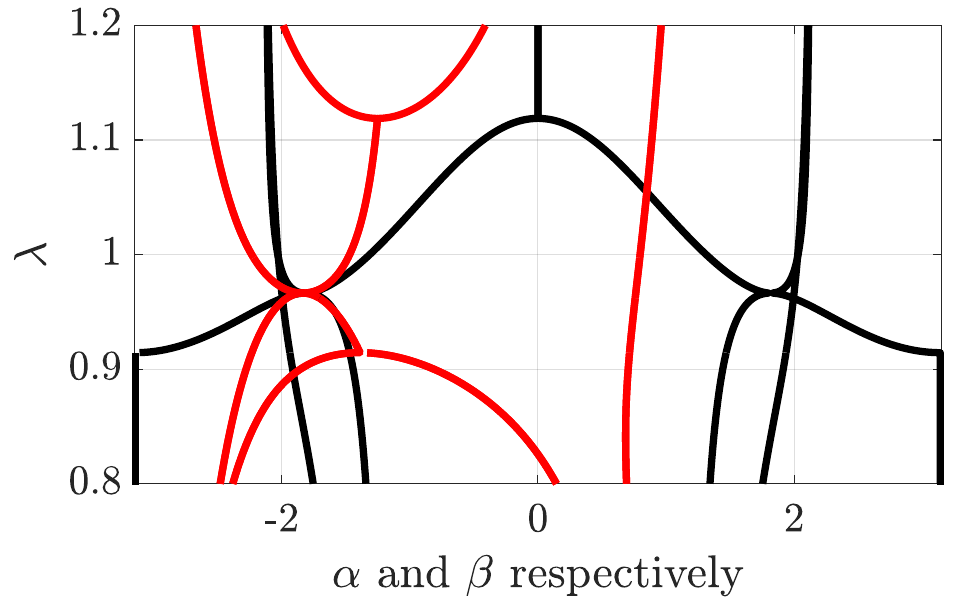}}\quad
    \subfloat[][The set $\mathbf{Z}(f_m)$ passes over complex confluent roots, and therefore does not form a closed polar curve. ]%
    {\includegraphics[width=0.48\linewidth]{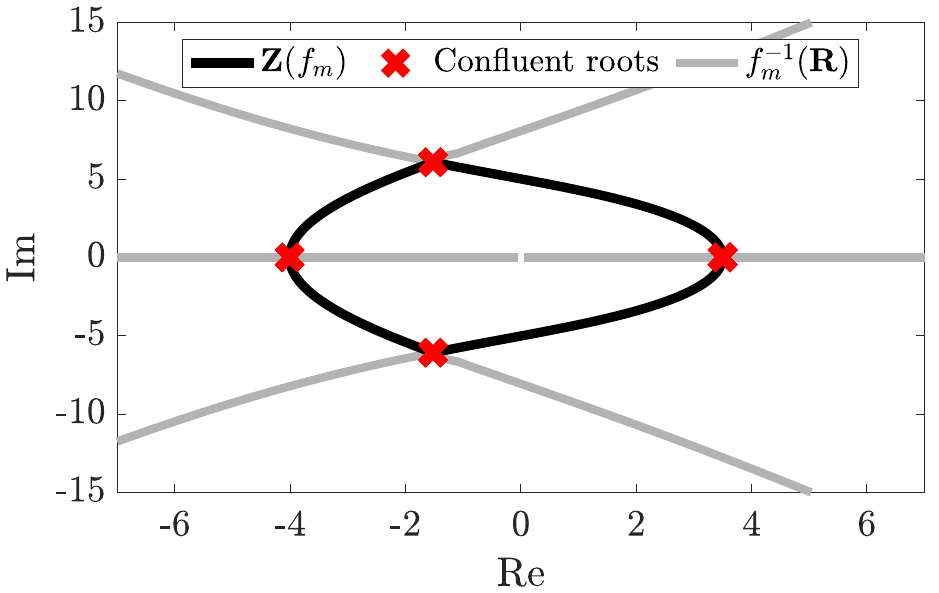}}
    \caption{The set $\mathbf{Z}(f_m) \subseteq f_m^{-1}(\R)$ is no longer parametrised by a polar curve, as by Remark \ref{rem: real and imag confluent roots}, the complex confluent root is not in $\mathbf{Z}(f_m)$. Therefore $\mathbf{Z}(f_m)$ consists of two disjoint algebraic curves and are thus not spanned by a polar curve. The touching of the complex bands in (A) indicates the presence of confluent complex roots. Computation performed for  $m = 3$, $\kappa = 2.3052155325$, $\rho = 6.5$.}
    %p = 2.305215532501812
    \label{Fig: Preimage confluent roots}
\end{figure}  

If confluent roots were permitted, one could, at least in theory, obtain a situation like the one shown in Figure \ref{Fig: counterexample confluent roots}. We stress that the symbol function and the open spectrum shown are schematic illustrations rather than plots of a specific symbol function. Their purpose is solely to demonstrate potential issues in the confluent case, where a bifurcation in the complex plane may occur.

\begin{figure}[h]
    \centering
    \subfloat[][Solid blue line is the open limit and the dashed line corresponds to $f(r\mathbb{T})$.]%
    {\includegraphics[width=0.45\linewidth]{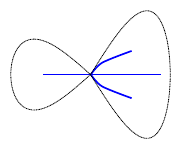}}\qquad\qquad
    \subfloat[][Solid blue line is the open limit and the dashed line corresponds to $f(\tilde{r}\mathbb{T})$.]%
    {\includegraphics[width=0.45\linewidth]{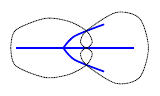}\vspace{8mm}}
    \caption{A schematic counterexample illustrating a potential failure point for complex-valued confluent roots, specifically showing how the key argument in the proof of Theorem \ref{Thm: real sectra criterion} can break down even when $\alpha(\lambda)$ is continuous. Possibly, this could lead to a bifurcation of $\sigma_{\mathrm{open}}$ into the complex plane.}
    \label{Fig: counterexample confluent roots}
\end{figure}

\begin{figure}[h]
    \centering
    \subfloat[][Solid blue line is the open limit and the dashed line corresponds to $f(r\mathbb{T})$ for $r = 4.97$.]%
    {\includegraphics[width=0.45\linewidth]{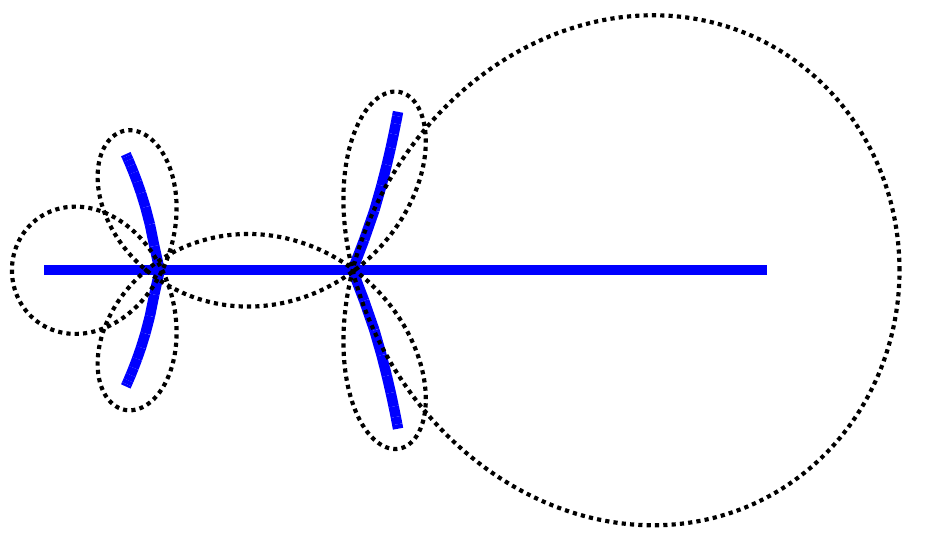}}\qquad\qquad
    \subfloat[][Solid blue line is the open limit and the dashed line corresponds to $f(r\mathbb{T})$ for $r = 4.7$.]%
    {\includegraphics[width=0.45\linewidth]{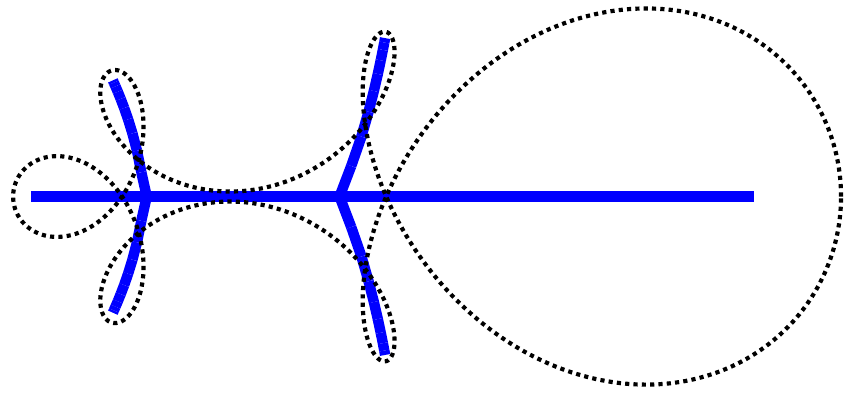}\vspace{4mm}}
    \caption{From visual inspection, it is immediately apparent that $\alpha(\lambda)$ is not continuous at the point where the open limit branches into the complex plane. However, if the set $\mathbf{Z}(f_m)$ is parametrised by a polar curve according to \eqref{eq: polar curve bijection}, then $\alpha(\lambda)$ is known to be continuous, which rules out the appearance of such branching behaviour as shown in this figure. Computation performed for $m = 3$, $\kappa = 1$ and $\rho = 6$.}
    \label{Fig: Continuous alpha curve}
\end{figure}

The next result numerically illustrates that the assumptions in Theorem \ref{Thm: real sectra criterion} are stronger than necessary. Namely, the assumption of Conjecture \ref{conjecture: real spectra} that $f_m^{-1}(\R)$ contains a Jordan curve seems to be optimal, this is illustrated in Figure \ref{Fig: open spectrum double jordan curve}.

\begin{figure}[h]
    \centering
    \subfloat[][Complex band structure with non-monotone $\alpha(\lambda)$ band.]%
    {\includegraphics[width=0.48\linewidth]{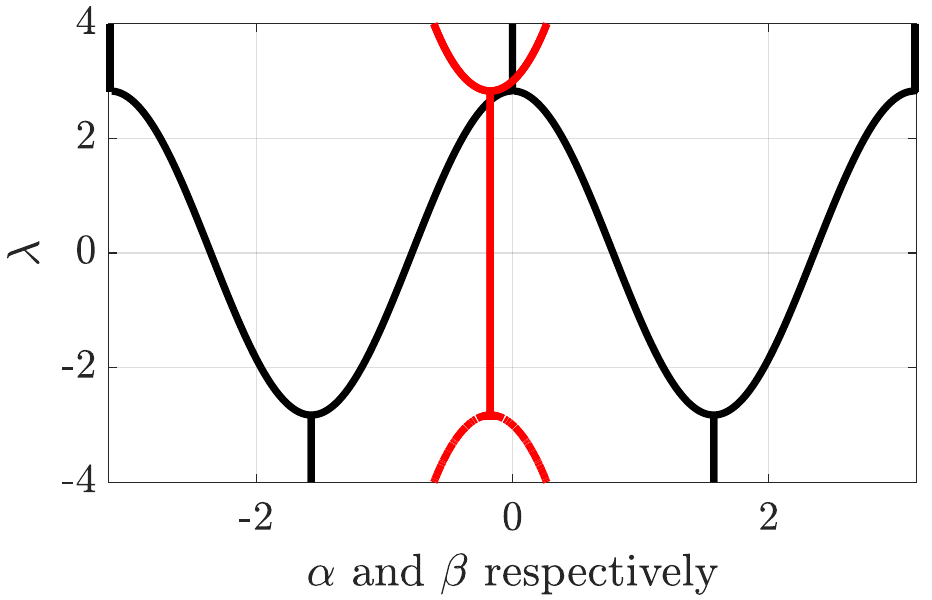}}\quad
    \subfloat[][The complex valued confluent roots are not part of the set $\mathbf{Z}(f_m)$.]%
    {\includegraphics[width=0.48\linewidth]{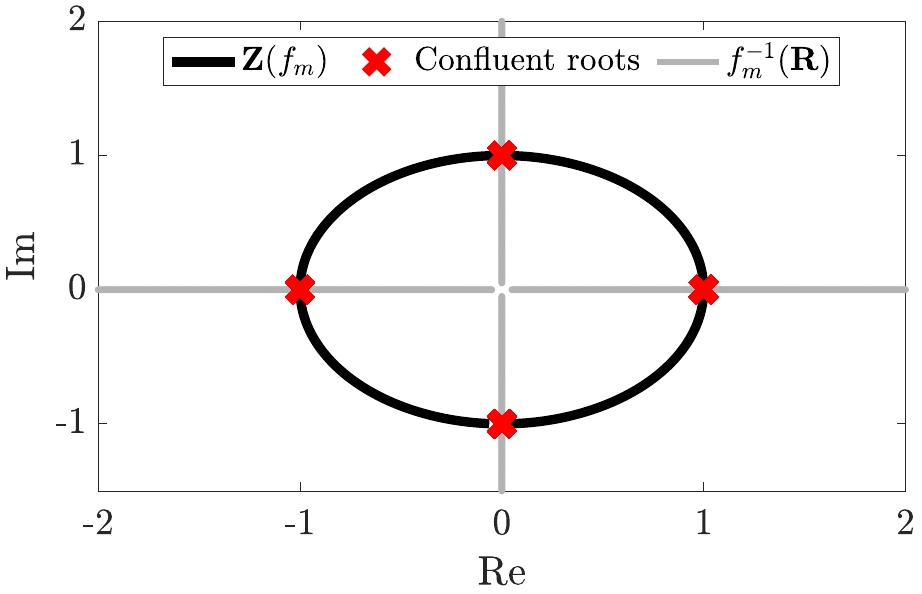}}
    \caption{The complex-valued confluent points do not belong to the set $\mathbf{Z}(f_m)$ introduced in \eqref{set: non confluent CBS}. As a consequence, $\mathbf{Z}(f_m)$ is not generated by a polar curve, and Theorem \ref{Thm: real sectra criterion} does not apply. Nonetheless, the open spectrum remains real. Computation performed for a symbol function $f_m(z) = z^{-2} + 2z^2$.}
    \label{Fig: open spectrum double jordan curve}
\end{figure}

\section{Numerical Analysis}\label{sec: Numerical analysis}
In this section, we present the numerical procedure employed to compute the Fourier coefficients $(\hat{f\circ p})_k$.
For a given range, we compute the roots of $z^m\bigl(f_m(z)-\lambda\bigr)$ sorted by ascending magnitude and keep the roots if $|z_m(\lambda)| = |z_{m+1}(\lambda)|$. The set of roots that satisfy this condition constitutes a discrete sampling of a continuous curve $p(e^{\i\theta})$, where the sampling of the curve is given by non-uniform angles $\theta_k$ for $1\leq k \leq N$. Then the Fourier coefficients are evaluated by means of the Fast Fourier transform (FFT) with non-uniform sampling.
By the Nyquist–Shannon sampling theorem for the Fast Fourier transform, the sample size $N$ must be at least twice the bandwidth of the signal to avoid aliasing.
Because the sample size remains fixed (and does not scale with the size of the matrix), entries outside a certain bandwidth will eventually be polluted due to aliasing. This is the reason we discard the high-frequency coefficients of $(\hat{f\circ p})_k$.
This can be formally established by the following result,
\begin{corollary}\label{Cor: aliasing Fourier coefficients}
    Consider the Hermitian Toeplitz matrix $\mathbf{T}_n(f\circ p)$ as well as the $t$-banded approximation $\mathbf{T}_n\bigl((f\circ p)_t\bigr)$, and let $\lambda_k$ and $\tau_k$, for $1\leq k \leq n$ be the sorted eigenvalues, then it holds that
    \begin{equation}
        | \lambda_k - \tau_k| = \mathcal{O} \bigr(\sum_{|d| \geq t}(\hat{f\circ p})_d \bigr).
    \end{equation}
\end{corollary}

\begin{proof}
    Let us start by writing
    \begin{equation}
        \mathbf{T}_n\bigl((f\circ p)_t\bigr) + \mathbf{E}_t = \mathbf{T}_n(f\circ p),
    \end{equation}
    where $\mathbf{E}_t$ is a Toeplitz matrix comprising the coefficients $(\hat{f\circ p})_k$, $t \leq |k| \leq n$. As the Fourier coefficients $(\hat{f\circ p})_{|k|}$ are decaying in $k$ we may estimate the norm,
    \begin{equation}
        \lVert \mathbf{E}_t \rVert_2 = \mathcal{O}\bigl(\sum_{|d| \geq t}(\hat{f\circ p})_d \bigr).
    \end{equation}
    The result then directly follows By Weyl's theorem on the spectral stability which completes the proof.
\end{proof}

\begin{figure}[h]
    \centering
    \subfloat[][Fourier coefficients computed for $100$ sampling points in the FFT.]%
    {\includegraphics[width=0.48\linewidth]{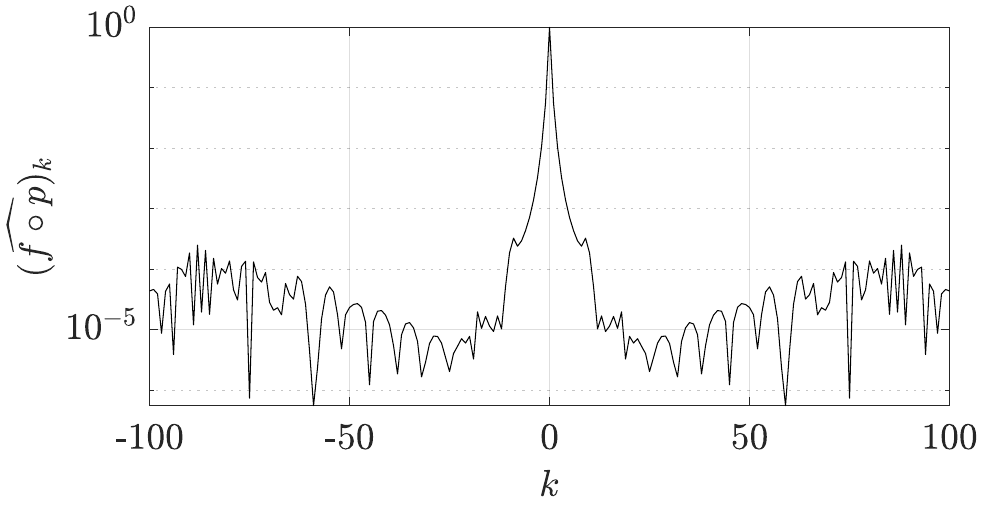}}\quad
    \subfloat[][Fourier coefficients computed for $1000$ sampling points in the FFT.]%
    {\includegraphics[width=0.48\linewidth]{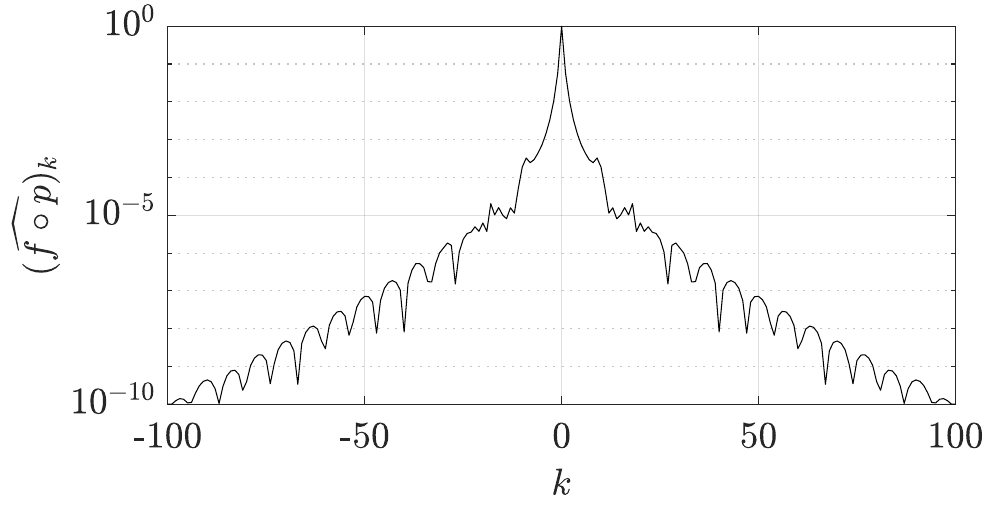}}
    \caption{Once the bandwidth exceeds a certain threshold, the numerically evaluated Fourier coefficients begin to oscillate instead of decaying, which motivates the need for Corollary \ref{Cor: aliasing Fourier coefficients}. }
    \label{Fig: Fourier coefficents composition}
\end{figure}

For analytic functions, the FFT is known to be exponentially convergent. However, for functions that are $C^k$, the numerical error is expected to be $\mathcal{O}(N^{-k})$.

\section{Derivation of the similarity transform}\label{Appendix: Similarity transform}
In this section, we will formally derive a similarity transform presented in Section \ref{Sec: Asymptotic Similarity} between $\mathbf{T}_n(f)$ and $\mathbf{T}_n(f\circ p)$ for general polar curves $p: \mathbb{T}\to\C$.
For the doubly infinite Toeplitz matrix $\mathbf{T}_n(f)$, commonly referred to as the Laurent operator, in the limit as $n\to\infty$ it holds
\begin{equation}
    \mathbf{T}_n(f) \approx \mathbf{M}^{-1}\mathbf{T}_n(f\circ p) \mathbf{M},
\end{equation}
implying that the matrices are asymptotically similar, where the entries of $\mathbf{M}$ are given by
\begin{equation}
(\mathbf{M})_{i, j} = \frac{1}{2\pi}\int_0^{2\pi} p(e^{\i \theta})^{-j} e^{\i\theta i} \d \theta, \quad i,j\in \Z,
\end{equation}
and correspond to the $i$-th Fourier coefficient of $p:\mathbb{T}\to\C$ raised to the $j$-th power.
To see this, consider the $m$-banded Toeplitz matrix $\mathbf{T}_n(f)$ that is generated by the sequence $\{a_k\}_{k\in \Z}$ such that
    \begin{equation}
        f(e^{\i\theta}) = \sum_{k = -m}^m a_k e^{-\i\theta k},\quad\text{and}\quad f\bigl(p(e^{\i \theta})\bigr) = \sum_{k = -m}^m a_k p(e^{\i \theta})^{-k}.
    \end{equation}
    Let us start by noting that the transformed Toeplitz matrix $\mathbf{T}_n(f\circ p)$ is generated by the sequence
    \begin{equation}
(\hat{f\circ p})_j = \frac{1}{2\pi}\int_0^{2\pi} f\bigl(p(e^{\i\theta})\bigr) e^{\i j \theta} \d \theta 
= \frac{1}{2\pi}\int_0^{2\pi} \sum_{k= -m}^m a_k p(e^{\i\theta})^{-k} e^{\i j \theta} \d\theta 
= \sum_{k= -m}^m a_k \cdot \frac{1}{2\pi}\int_{0}^{2\pi} p(e^{\i\theta})^{-k} e^{\i \theta j} \d \theta,
    \end{equation}
    As a result, it follows that,
    \begin{equation}
        c_n = \sum_{k = -m}^m
        (\mathbf{M})_{n, k} a_k, \quad\text{and}\quad f\bigl(p(e^{\i \theta})\bigr) = \sum_{k\in \Z} c_k e^{-\i\theta k}.
    \end{equation}
    We have that
    \begin{equation}\label{eq: before sim}
        \bigl(\mathbf{M}\mathbf{T}(f) \bigr)_{i,j} = \sum_{k \in \Z} (\mathbf{M})_{i,k}\bigl(\mathbf{T}(f)\bigr)_{k,j}  = \sum_{k\in \Z} (\mathbf{M})_{i,k} a_{k-j} = \sum_{k'= -m}^m (\mathbf{M})_{i, k' + j}a_{k'},
    \end{equation}
    as well as
    \begin{equation}\label{eq: after sim}
        \bigl(\mathbf{T}(f\circ p) \mathbf{M}\bigr)_{i,j} = \sum_{k \in \Z} \bigl(\mathbf{T}(f\circ p)\bigr)_{i,k} (\mathbf{M})_{k, j} = \sum_{k'= -m}^m \left( \sum_{k \in \Z} (\mathbf{M})_{i-k, k'}(\mathbf{M})_{k, j} \right) a_{k'}.
    \end{equation}

    The only thing left to verify is the equality between \eqref{eq: before sim} and \eqref{eq: after sim}. Since $(\mathbf{M})_{\cdot, j}$ is the Fourier coefficient sequence of $p^{-j}$ and $p^{-k'}p^{-j} = p^{-(k'+j)}$, the convolution theorem yields
    \begin{equation}
       \sum_{k \in \Z} (\mathbf{M})_{i-k, k'}(\mathbf{M})_{k, j} = \bigl(\widehat{p^{-k'}} * \widehat{p^{-j}}\bigr)_i = \bigl(\widehat{p^{-(k'+j)}}\bigr)_i = (\mathbf{M})_{i, k'+j},
    \end{equation}
    which is precisely what is required. We stress that asymptotic similarity alone is still insufficient to ensure pointwise convergence of the eigenvalues. Truncating the doubly infinite problem to a Toeplitz matrix with $\mathbf{P}_n$ and writing $\mathbf{Q}_n = \mathbf{I} -\mathbf{P}_n$, with $\mathbf{M}_n = \mathbf{P}_n\mathbf{M}\mathbf{P}_n$. The exact finite-section defect reads
    \begin{equation}\label{eq: edge projection}
        \mathbf{T}_n(f)\mathbf{M}_n - \mathbf{M}_n\mathbf{T}_n(f\circ p) = \mathbf{P}_n\mathbf{L}(f)\mathbf{Q}_n\mathbf{M}\mathbf{P}_n - \mathbf{P}_n\mathbf{M}\mathbf{Q}_n\mathbf{L}(f\circ p)\mathbf{P}_n =: \mathbf{E}_n
    \end{equation}
    The term $\mathbf{P}_n\mathbf{L}(f)\mathbf{Q}_n\mathbf{M}\mathbf{P}_n$ has rank less than $2m$, while the term $\mathbf{P}_n\mathbf{M}\mathbf{Q}_n\mathbf{L}(f\circ p)\mathbf{P}_n$ is exponentially localised in the two corners. Consequently \eqref{eq: edge projection} is mostly due to boundary effects which average out in the limit $n\to \infty$ and therefore precisely captures the convergence of the density of states.    
    Indeed as $\lvert \mathbf{E}_n\rvert \xrightarrow{n\to\infty} 0$, we know from Corollary \ref{cor: Asymptotic equivalence and DoS} that their densities of states coincide. However, this does not generally yield pointwise control of the eigenvalues.    
\end{document}